\documentclass[10pt,reqno]{amsart}

\usepackage [english]{babel}

\usepackage{dsfont}
\usepackage{xcolor}

\usepackage{amsmath}
\usepackage{amsfonts}
\usepackage{amssymb}
\usepackage{amsthm}
\usepackage{mathrsfs}
\usepackage{latexsym}

\usepackage{cite}

\usepackage{esint}

\numberwithin{equation}{section}

\usepackage{graphicx,subfigure}

\usepackage{ifthen} 

\usepackage{comment}

\provideboolean{shownotes} 
\setboolean{shownotes}{true} 
\newcommand{\margnote}[1]{
\ifthenelse{\boolean{shownotes}}%
{\marginpar{\raggedright\tiny\texttt{#1}}}%
}

\newcommand{\hole}[1]{
\ifthenelse{\boolean{shownotes}}%
{\begin{center} \fbox{ \rule {.25cm}{0cm}
\rule[-.1cm]{0cm}{.4cm} \parbox{.85\textwidth}{\begin{center}
\texttt{#1}\end{center}} \rule {.25cm}{0cm}}\end{center}}
{}
}

\usepackage[hidelinks]{hyperref}

\theoremstyle{plain}

\newtheorem{lemma}{Lemma}[section]
\newtheorem{theorem}[lemma]{Theorem}
\newtheorem{proposition}[lemma]{Proposition}
\newtheorem{corollary}[lemma]{Corollary}

\theoremstyle{definition}
\newtheorem{remark}[lemma]{Remark}
\newtheorem{definition}[lemma]{Definition}

\theoremstyle{remark}

\newcommand{\A}{\mathbb{A}}

\newcommand{\R}{\mathbb{R}}

\begin{document}

\title[Orbital Stability of Optical Solitons]{Orbital Stability of Optical Solitons in 2d}
\author[S. Moroni]{Sergio Moroni}
\address{{\rm (S. Moroni)} BCAM - Basque Center for Applied Mathematics, Alameda de Mazarredo 14, 48009, Bilbao, Bizkaia, Spain}
\email{smoroni@bcamath.org}

    \begin{abstract}
    We present a stability result for ground states of a Schrödinger-Poisson system in $(2+1)$ dimension, modelling the propagation of a light beam through a liquid crystal with nonlocal nonlinear response. The core of the proof is a coercivity bound on the second derivative of the action, where non scaling nonlinearities and the coupled system present the major difficulties. In addition we prove existence of a ground state with frequency $\sigma$ for any $\sigma \in (0,1)$ as a minimal point over an appropriate Nehari manifold.
    \end{abstract}
   \maketitle

\section*{Introduction}
Optical properties of nematic liquid crystals have received great attention in the last years, as they can support stationary optical waves, of large interest both in theory and in applications. Heuristically, when a light wave propagates through a nematic liquid crystal, its electric field induces a dipolar polarization in the anisotropic medium. The electromagnetic action of the dipoles cause a reorientation of the molecules in the liquid crystal, and hence a modification of the light refractive index of the material. Due to high susceptibility  of nematic liquid crystals, the response is nonlocal, meaning that has effects far beyond the region occupied by the light wave, and nonlinear. This response has a self-focusing effect on the light beam, supporting waveguides that counterbalance the diffraction spreading nature of light beam, and, in optimal shapes, allows the existence of stationary waves. \\The interested reader is referred to \cite{as1} or \cite{as2}, for a physical overview of the topic and a presentation of the main experiments in the field, or to \cite{lib}, Chapters 2 and 6, for a wider mathematical introduction.\bigskip
\\
In this paper we study the ground states, proving orbital stability, existence for any frequency $0<\sigma<1$ and a decay estimate, of the Schrödinger-Poisson system
\begin{align}
     &i\partial_z u + \frac{1}{2}\Delta u + u \sin (2\theta)=0  \label{eqcal}
 \\&   -\lambda \Delta \theta + q\sin (2\theta)= 2|u|^2 \cos(2\theta) \label{eqcal2} 
\end{align}
in dimension $(2+1)$. The axis $z$, referred to as the optical axis, is the direction of the propagation of a light beam, while $\Delta$ is the Laplacian in the transverse coordinates $(x,y)$.\\
The system models the propagation of a laser beam through a planar cell filled with a nematic liquid crystal, oriented by an external electric field $E$. Equation \eqref{eqcal} represents the evolution of the light beam, with $u:\R^2 \to \mathbb{C}$ the complex amplitude of the electric field, while \eqref{eqcal2} is the nonlocal response of the medium, with $\theta:\R^2 \to \R$ the director field angle of the light-induced reorientation. The values $q, \lambda$ are positive constants depending, respectively, on the intensity of the pre-tilting electric field and on the elastic response of the medium, that is on its property of nonlocality.
In \cite{Arg} a heuristic derivation of the equations is presented in the Appendix, while \cite{as1}, \cite{as2} and the references therein give a deeper understanding of the system and of the related observed phenomena.
\bigskip

The system was rigorously studied in \cite{Arg}, where the authors proved global existence and regularity for the Cauchy problem, and existence of stationary waves as minimizers, over couples $(u,\theta)$ with $L^2$ norm of $u$ fixed, of the Hamiltonian:
\begin{equation}
    \label{ene}
    E(u,\theta):= \frac{1}{4}\int_{\R^2} |\nabla u|^2 + \lambda |\nabla \theta|^2 - 2 |u|^2 \sin (2\theta) + q (1-\cos(2\theta)) \,dx
\end{equation}
A minimal configurations $(v,\phi)$ satisfy the equations 
\begin{align}
    -\Delta v& + 2\sigma v -2v\sin (2\phi)=0
    \label{sigma}
    \\ -\lambda \Delta \phi & +q \sin (2\phi) -2|v|^2 \cos(2\phi)=0 
\label{sig}
\end{align} 
where $\sigma\in \R$ is the Lagrange multiplier. The couple $(e^{i\sigma z}v(x,y), \phi(x,y))$ is then a stationary wave for the system \eqref{eqcal}-\eqref{eqcal2} as it evolves along the optic axis changing only by a phase shift of frequency $\sigma$. The paper does not state uniqueness for the ground state over the constraint.
\bigskip\\
We will present a first stability result for those stationary waves. Loosely speaking, a stationary wave is stable if the evolution through equations \eqref{eqcal}-\eqref{eqcal2} of an initial datum close to $(v, \phi)$ remains close to the orbit of the ground state. 
\\Proving stability or instability of stationary waves for a dynamical system has attracted considerable effort in the scientific community; we refer here the classical works \cite{StS}, \cite{W}, \cite{W2}, \cite{S} \cite{SS}, or the recent reports \cite{Rep1}, \cite{Rep2}. In particular, a stability result provides a strong justification for the relevance of
 the mathematical model to applications, as only locally stable solutions are expected to be seen in 
experiments and numerical simulations. 
\bigskip\\A possible path to a stability result is strictly variational. Ground states are often proved to exist as minimizers of the energy over a constraint, possibly  invariant with respect to the dynamic. The common difficulty in this passage revolves around getting \textit{at least one} minimizing sequence with enough compactness to guarantee convergence to a minimizer. On the other hand, if the variational method employed is robust enough to ensure convergence of \textit{any} minimizing sequence to a minimizer, stability follows readily. This approach was first conceived in \cite{sss}, and later used in many different settings (\cite{SP1}, \cite{SP2}, \cite{Luo}, \cite{SP4}, \cite{Din}). \\We point out that in the previous paper \cite{aaa}, for a simplified model of \eqref{eqcal}-\eqref{eqcal2} the authors were able to prove compactness with the required generality, and hence stability. This does not hold for solution of \eqref{sigma}-\eqref{sig}, as the variational method employed in \cite{Arg} improves compactness of a minimizing sequence through rearrangements. Similarly, orbital stability is not yet proved for ground states of a more complicated model, with higher order nonlinear effects taken into account, studied in \cite{Arg2}.
\bigskip\\A different stratagem adapts Lyapunov stability to Hamiltonian systems in Hilbert spaces. One defines a proper action $S$ which is preserved by the evolution law and with its first derivative nullifying in the ground state. Coercivity of the second derivative of the action then implies orbital stability. 
\\The implementation of this idea for stability of stationary states dates back to \cite{StS}, \cite{W}, \cite{W2}; and it was later extensively exploited: see \cite{DeB}, \cite{Ir}, \cite{Fuk1}, \cite{Kom}, \cite{Ad}, \cite{Jea}.
\bigskip\\The main goal of this paper is proving stability by the second approach. The reason for this is twofold. On the one hand, it is mathematically more challenging: we have to carry out an analysis of the spectrum of $S''$, which contains a number of difficulties. On the other, we regard this stability result as more complete. It provides a better description of the stable orbit, which the evolution of a perturbed datum remains close to. Moreover, gaining a deep understanding of the spectra of the linearized operator is necessary for studying stronger notions of stability.\bigskip
\\We give two definitions of orbital stability, which reflect the two different approaches, and state our main result.
\begin{definition}
    \label{defst}
Let $(v,\phi)$ be a ground state, solving \eqref{sigma}-\eqref{sig} for a certain $\sigma$. We say that it is orbitally stable if for every $\varepsilon>0$ there exists a $\delta>0$ such that  
 \begin{equation}
      \label{eqst}
    \|u_0 - v\|_{H^1} \le \delta \longrightarrow \sup_{t} \inf_{\substack{y\in\R^2, \alpha \in \R\\ (w, \psi) \in \mathcal{M}^\sigma}} \left \| (u,\theta)(t) - \left ( e^{i\alpha} w(\cdot-y), \psi(\cdot - y) \right) \right \|_{H^1\times H^1} <\varepsilon
  \end{equation}
  where $\left (u, \theta)(t)\right )$ is the solution of \eqref{eqcal}-\eqref{eqcal2} with initial condition $u_0$ and $\mathcal{M}^\sigma $ is defined as
  \begin{equation}
  \label{DefM}
      \mathcal{M^\sigma} := \left \{ (w,\psi) \in \left (H^1_{rad}\right )^2  \ \biggm| \begin{array}{l}
 2E(v,   \phi) + \sigma \int |v|^2= 2E(w,\psi) + \sigma \int |w|^2\\\
(w, \psi ) \ \mathrm{solve \ } \eqref{sigma}-\eqref{sig} \mathrm{ \ for \ the \ given \ } \sigma
\end{array}
     \right \} 
  \end{equation}
\end{definition}
We will prove in Section \ref{secstab} that the set $M^\sigma$ collects all the radial real valued ground states with the same frequency $\sigma$ and $L^2$ norm. Up to symmetries, it represents the orbit around which stability is proved. For the moment we stress that by our definition $\mathcal{M}^\sigma$ always contains at least a ground state.
\begin{definition}
    \label{defst2}
Let $(v,\phi)$ be a ground state over the constraint 
  \begin{equation*}
    S_a:= \left \{ (u, \theta) \in H^1 \times H^1 \ \ | \ \ \|u\|_{L^2}^2 = a \right \}
\end{equation*}We say that it is variationally stable if for every $\varepsilon>0$ there exists a $\delta>0$ such that  
 \begin{equation}
      \label{eqstv}
      \|u_0 - v\|_{H^1} \le \delta \longrightarrow \sup_{t} \inf_{(\tilde v, \tilde \phi) \in \Sigma_a} \left \| (u,\theta)(t) -  ( \tilde v, \tilde \phi) \right \|_{H^1\times H^1} <\varepsilon
  \end{equation}
  where $\Sigma_a$ is defined as
  \begin{equation}
      \Sigma_a := \left \{ (w,\psi) \in S_a \ | \ E(w, \psi ) = E(v, \phi) \right \}
  \end{equation}
\end{definition}
\begin{theorem}
\label{teost}
    Let $a_0$ be such that for any $a>a_0$, there exists a ground state $(v_a, \phi_a) $ over $S_a$. Then for almost every $a>a_0$ the ground state is orbitally stable.
\end{theorem}
The work of \cite{Arg} assures the existence of $a_0$ as in the Theorem (see Theorem \ref{min} below). 
\\Once we can get strict coercivity of the bilinear form associated to $S''$ evaluated at a ground state, the proof is a simple adaptation of  \cite{StS}, or \cite{Fuk}. The difficulties are hence concentrated on this bound, stated in Proposition \ref{S''}.
\\$S''\ge 0$ is direct, as $(v,\phi)$ is a minimizer over $S_a$, up to natural restrictions linked to the presence of the constraint. Improving this to a coercivity control contains two major difficulties. 
\bigskip \\
The first difficulty comes from the Kernel of $S''$, as coercivity is clearly false for $0$-eigenvalues of the operator. We would like to prove that the Kernel of $S''$ is spanned by some "symmetry generators", which can be neglected by proper modulation. Invariances of the system guarantee that generators of translation and complex phase shift are in the Kernel, and their modulation is well understood.\\
This argument leaves uncovered the radial, real valued part of the Kernel. In the literature, nonexistence of similar eigenfunctions is often proved via uniqueness for the ground state equation. The reader is specifically referred to \cite{W}, Appendix A, for a clear exposition of the implication, or to \cite{DeB}, where  uniqueness is used in combination with topological argument on the Morse degree. We stress that in \cite{Jea}, \cite{Fuk1}, \cite{Kom} the authors prove stability with the same scheme, in a context where uniqueness is not known. Still, their strategy relies indirectly on it, as they consider frequency $\sigma$ for which the ground state, up to rescaling, converges to the unique and nondegenerate stationary wave of a simpler equation. 
\bigskip 
\\The question of uniqueness for radial positive ground state equations is classical and has received significant attention.
 Without claiming to be exhaustive, we direct the reader to \cite{Cof}, \cite{K}, where oscillation theory for ODE's is the key ingredient, or the more modern \cite{Y}, \cite{SW1}, \cite{SW2}, which rely on more intricate identities satisfied by the solutions. \\
 We are not able to prove uniqueness for \eqref{sigma}-\eqref{sig}. The method recalled above do not easily adapt to our case: oscillation theory is particular hard in the case of a system, and the trigonometric nonlinearities make all the delicate equalities far more complicated. Moreover uniqueness does not follow easily by the minimization problem, since the energy $E$ is not convex.
\\To overcome this problem, we consider in Definition \ref{defst} a larger orbit with all the ground states with the same frequency $\sigma$. If uniqueness held, the set would coincide with the ground state only; but where it is not known, it maintains the implicit definition above. 
\bigskip \\The second major difficulty concerns the coercivity bound. The information available at this moment for the self-adjoint operator $S''$ (namely, non-negativity and orthogonality to the kernel) is insufficient.\\
Loosely speaking, coercivity of the action serves as a selection criterion for the frequency of the dynamical orbit of a perturbed initial datum. In a number of work (see \cite{W2}, \cite{FA}, \cite{FA0}, \cite{StS}, \cite{New}, \cite{KT}), the condition for stability can be expressed as
\begin{equation}
\label{dp}
\frac{d}{d\sigma} \|\varphi_\sigma\|_{L^2}^2 > 0,
\end{equation}
where $\varphi_\sigma$ denotes the ground state corresponding to frequency $\sigma$. This formulation is essentially equivalent to the coercivity of $S''$, but better clarifies the heuristic of this step. The positivity of this derivative naturally indicates stability, while a negative value has been used as a criterion for instability (see \cite{FA0}, \cite{FA}, \cite{FukO}, \cite{OO}, \cite{KT}).  
\bigskip\\
In the stable case, the strict inequality ensures that the evolution of a perturbed datum selects the orbit associated with the initial ground state, without any change in frequency. By conservation laws, a solution originating from an initial datum close to $\varphi_\sigma$ preserves both the energy and the $L^2$ norm, remaining close to those of $\varphi_\sigma$. Heuristically, a zero derivative in place of \eqref{dp} implies existence of ground states with different frequencies but almost identical $L^2$ norms; in this case, conservation of flow-invariant quantities does not single out a specific $\sigma$. In \cite{W2}, \cite{WC}, the author investigates stability for critical power nonlinearities, which correspond to the $0$ derivative case. The desired coercivity is recovered by adding to the stable orbit dilations of the ground state, which do not preserve the frequency.\bigskip
\\
The condition \eqref{dp} is often checked through a Pohozaev identity and scaling properties of the equation, which allows to explicitly calculate the derivative ( see \cite{New}, \cite{FA}, \cite{W2}, or \cite{KT} for the case of double power nonlinearity). In our situation, lack of an easy scaling for the nonlinearities means less useful integral identities. Still, energy comparison implies a monotone dependence, and with some modifications to the proof of \cite{W2}, we can implement this weaker condition in the spectral analysis. Stability result follows for almost every value of the charge constraint $a$, reflecting the a.e. differentiability of a monotone function.
\bigskip\\The result of Theorem \ref{teost} is weaker than more standard stability results, due to the presence of the implicit $\mathcal{M}^\sigma$ and the a.e. limitation on $a$. Nevertheless, we consider it the major novelty of the paper as the strategies for the specific difficulties listed above are new; moreover to the best of our knowledge it is the first spectral stability result for a Schrödinger-Poisson system. The interested reader can compare with other proofs for coupled systems present in the literature: \cite{SP1}, \cite{SP2}, \cite{SP3}, \cite{SP4}, \cite{D2}.
\bigskip\\ 
To complete the stability study, we  prove the following result, covering also the missing ground states.
\begin{proposition}
Let $a_0$ be as in Theorem \ref{teost}. For any $a>a_0$ the ground state over $S_a$ are variationally stable.     
\label{teost2}
\end{proposition}
The result is an adaptation of the concentration-compactness method to a coupled system, following \cite{aaa}. \bigskip \\
In the third result of our paper, we prove the existence of a stationary wave for the system for any frequency value $\sigma \in (0,1)$:
\begin{theorem}
\label{teosig}
    For any $0<\sigma<1$ there exists a radially symmetric and decreasing stationary wave $(v_\sigma,\phi_\sigma)\in H^1_{rad}\times H^1_{rad}$ that is a solution of \eqref{sigma}-\eqref{sig} with the fixed value of $\sigma$.
\end{theorem}
The Theorem states that the models \eqref{eqcal}-\eqref{eqcal2} provides the existence of a stationary wave with frequency $\sigma $ for any $0<\sigma <1$. Providing a relevant physical meaning for this family of stationary waves goes beyond the purpose of this article; specifically, this paper does not contain a stability result for them. Still, the parameter $\sigma$ appears as physically relevant in Theorem \ref{min}, and it is mathematically interesting to prove existence or non existence results for singular solutions with respect to a moving parameter.  
\bigskip \\
The idea of the proof is to get the existence of the stationary wave as minimum points for the energy. This cannot be done directly with the energy $E$ defined in \eqref{ene}, as it is easily seen to be unbounded below. \\In Theorem \ref{min} from \cite{Arg}, the authors get the existence of a minimum looking to the constrained problem with $L^2$ norm fixed. On the one hand this leads quite naturally to the existence of a minimizer, but on the other $\sigma$ has the role of a Lagrangian multiplier and there is no possibility to control its value a priori.
\bigskip\\We consider a modified energy with a term depending on $\sigma$, and minimize the functional over its Nehari manifold. The idea was firstly used by Nehari in \cite{Ne1}, \cite{Ne2} where he obtained some non trivial solution to specific nonlinear ODEs via a variational methods, and later had a huge application in critical point theory. We were mainly inspired by \cite{Nene}, \cite{Nena}.\bigskip\\In the proof of our result, apart from some technicalities linked to the low regularity the modified energy considered, the main difficulties come from lack of compactness for our problem. In particular, even with the restriction to the manifold, the energy is not coercive. \\But, by some energetic consideration, we are able to recover the desired compactness at the level of the minimal value, and hence conclude the Theorem. In this a crucial role is played by the interval of the parameter $0<\sigma<1$, as it leads to optimal competition between the two terms of opposite sign 
\begin{equation*}
    2\sigma\int |v|^2 - 2\int|v|^2 \sin (2\Theta(v))
\end{equation*}
\begin{remark}
    The limitation on the values of $\sigma$ emerges naturally from the equation. It is easy to show that there is no nontrivial solution of \eqref{sigma} for $\sigma \ge 1$: multiplying by $v$ and integrating by parts would lead
    \begin{equation*}
        \int |\nabla v|^2 + k \int |v|^2 \le 0
    \end{equation*}
    for a certain $k\ge 0$; hence $v=0$.
\\On the other hand, the bound $\sigma>0$ is related to the variational method used to prove existence. Again by equation \eqref{sigma}, $\sigma $ is linked to the part of the energy with a negative sign; a configuration satisfying equation \eqref{sigma} with $\sigma<0$ would have strictly positive energy and it would probably not appear as a minimum of a any constrained problem.
\end{remark}

At last we prove a regularity result and a decaying estimate for the stationary waves.
\begin{proposition}
    A radial decreasing solution $(v, \phi)$ of system \eqref{sigma}-\eqref{sig} belongs to $ H^k(\R^2) \times H^k(\R^2)$ for any $k$. and it decays exponentially at infinity; i.e. there exist constants $m, R, C$ such that for any $r>R$
    \begin{equation*}
        |v(r)| \le C e^{-mr}; \ \ \ \ |\phi(r)| \le C e^{-mr}
    \end{equation*}
    \label{decay}
\end{proposition}
The matter of the paper is organised as follows. In Section \ref{pre} we recall the exact Theorems from \cite{Arg} which we use subsequentially; moreover we improve some properties of the angle $\theta$ as solution of equation \eqref{eqcal2} and get a relation between charge constraints and frequency of the ground state. 
Section \ref{secstab} is dedicated to stability results Theorem \ref{teost} and Proposition \ref{teost2}. In Section \ref{secexi} we prove Theorem \ref{teosig} by a min–max method and finally, in the last section, we prove Proposition \ref{decay}.

\section{Preliminaries}\label{pre}
We start stating the precise results from \cite{Arg} needed in  this paper.
At first, we recall the existence Theorem for ground states.
\begin{theorem}
\label{min}
Let $J_a:= \inf_{S_a} E$. For $a$ above a certain threshold $a_0$, $J_a<0$ and is decreasing, while $J_a\equiv 0 $ for $a\le a_0$.\\For $a>a_0$ there exists a minimizer $(v,\phi)$ for $J_a$ and it satisfies the system \eqref{sigma}-\eqref{sig} for a real $\sigma$. Moreover we have $(v,\phi)\in H^1_{rad}\times H^1_{rad}$ decreasing, $v\ge 0$, $0\le \phi \le \pi/4$.\\Finally, there exists a $0<\tilde{a}\le a_0$ such that there is no $(v,\phi)$ solution to equations \eqref{sigma}-\eqref{sig} with $\|v\|_{L^2}^2\le \tilde{a}$.
\end{theorem}
Moreover they prove that if $u$ has some additional integrability, namely $u\in L^\infty \cap L^4$, it is possible to have a unique solution to \eqref{eqcal2} up to some technicalities.
\begin{theorem}
\label{teoang}
Given $u\in L^\infty\left (\R^2\right) \cap L^4\left(\R^2\right )$, there exists a unique $\theta \in H^2\left (\R^2\right ), \theta =\Theta(u)$ solution of \eqref{eqcal}
satisfying $0\le \theta < \pi/4$ and $\| \theta \|_{H^2} \le C\|u \|_{L^4}$.
\\Furthermore we have the following estimate
\begin{equation}
    \label{L4}
    \|\Theta(u_1) - \Theta(u_2)\|_{H_2} \le C_{q, \lambda}\left (\|u_1\|_{H^1}, \|u_2\|_{H^1}\right )  \left(1+ \|u_1\|_{L^\infty}^2 + \|u_2\|_{L^\infty}^2\right) \| u_1-u_2\|_{L^4}
\end{equation}
\end{theorem}
\begin{remark}
    As it was stated in the original paper \cite{Arg}, by the Gagliardo-Nirenberg inequality we can recast the $H^1$ norm on the right hand side, i.e. with the same hypothesis as in the Theorem \ref{teoang}, it holds 
\begin{equation}
    \label{stan}
    \|\Theta(u_1) - \Theta(u_2)\|_{H_2} \le C_{q, \lambda}\left(\|u_1\|_{H^1}, \|u_2\|_{H^1}\right)  \left(1+ \|u_1\|_{L^\infty}^2 + \|u_2\|_{L^\infty}^2\right) \| u_1-u_2\|_{H^1}
\end{equation}
\end{remark}
Finally, the authors prove global well posedness, uniqueness and regularity of the Schrödinger-Poisson system \eqref{eqcal}-\eqref{eqcal2} for an initial datum in the energy space.
\begin{theorem}    
\label{teoca}
    Given $u_0\in H^1\left (\R^2\right )$, there exists a unique $(u,\theta)\in C\left (\R, H^1\left (\R^2\right)\right ) \times L^\infty \left (\R, H^2\left (\R^2\right )\right )$ solution of the evolution problem \eqref{eqcal}-\eqref{eqcal2}
with initial datum $u_0$, such that $0\le \theta \le \pi/4$, $\nabla u \in L^4_{loc}(\R, L^4(\R^2))$.
Moreover, the quantities 
\begin{equation*}
    E(u,\theta)=\frac{1}{4}\int_{\R^2} |\nabla u|^2 + \lambda |\nabla \theta|^2 - 2 |u|^2 \sin (2\theta) + q (1-\cos(2\theta)) \,dx; \ \ \ Q(u):= \frac{\int_{\R^2}|u|^2}{2}
\end{equation*}
are preserved for all times. 
\end{theorem}
 The coupled system has a regularizing effect on $u$, as it propagates the initial $H^1$ regularity, and gets $\nabla u(t) \in L^4$ for a.e. $t$; for any such $t$ it follows $u(t)\in L^\infty $ and by Theorem \ref{teoang} exists $\theta(t)$ solution of \eqref{eqcal2}. We refer the interested reader to the original paper for more details.\\
Theorem \ref{min} recaps the results from Chapter 5 of \cite{Arg}, Theorem \ref{teoang} resumes Proposition 3.1 and  4.1, Theorem \ref{teoca} comes from Theorem 4.1 and 4.2  from the same paper. \\

We turn now to present an improvement of the estimate \eqref{stan}. In the proof of Theorem \ref{teoang}, the authors used the hypothesis $u\in L^\infty$ as their proof for the existence of $\Theta(u)$ relies on the continuation properties of a map on Banach spaces. More crucially, the $L^\infty$ norm appears explicitly in the Lipschitz control of the map, see \eqref{L4}, \eqref{stan}. \\We need to weaken this dependence on belonging to $L^\infty$ of $u$, as our arguments for Theorem \ref{teost} and in particular Theorem \ref{teosig} are strongly variational, and the $L^\infty$ space is not naturally related to a variational problem with the energy \eqref{ene}. 
\\We prove that we can naturally extend the definition of the angle $\Theta(u)$ to the space $u\in H^1$. 
Rather than using some density argument, we prefer to characterize $\Theta(u)$, as the unique minimizer of the energy $E$, for $u$ fixed. This variational property of $\Theta(u)$ has a key role in the proof of Theorem \ref{teosig}.
\begin{lemma}
\label{minan}
    For $u\in H^1(\R^2)$, there exists a $\tilde{\Theta}(u)\in H^1(\R^2)$, such that $\tilde\Theta (u) $ solves equation \eqref{eqcal2} and verifies $0\le \tilde\Theta(u) \le \pi/4$. 
    \\$\tilde\Theta(u) $ is characterized as the only minimizer of the functional 
    \begin{equation*}
        F_u(\theta) = \lambda\int|\nabla \theta|^2+ q\int 1-\cos 2\theta - 2\int|u|^2 \sin 2\theta
    \end{equation*}
    and for $u \in L^\infty\cap H^1$ it coincides with $\Theta(u)$ given by Theorem \ref{teoang}.\\The map $\tilde{\Theta}:H^1\to H^1$ is continuous; if $u$ is radially symmetric (resp. radially symmetric and decreasing), then $\tilde{\Theta}(u)$ is radially symmetric (resp. radially symmetric and decreasing).
    \\Finally, $\tilde{\Theta}(u)\in H^2$ and satisfies the estimate
    \begin{equation}
    \label{0an}
        \| \tilde{\Theta}(u)\|_{H^2} \le C\| u\|_{L^4}
    \end{equation}
\end{lemma}
\begin{proof}
    For $u\in H^1$ fixed, $F_u$ is bounded from below, coercive and lower semicontinuous with respect to weak convergence in $H^1$. It was observed in \cite{Arg}, see Lemma 5.1, that there exists a minimizing sequence $\theta_n$ satisfying $0\le \theta_n \le \pi / 4$.
    \\By standard variational methods $\theta_n \to \phi=: \tilde{\Theta}(u)$ minimum point for the energy, which solves the related Euler-Lagrange equation \eqref{eqcal2}. The convergence of $\theta_n \to \tilde{\Theta}(u)$ is strong in $H^1$ by convergence of the norms, and pointwise convergence implies $0  \le \tilde{\Theta}(u) \le \pi/4$. The minimum point is unique by convexity of $F_u$.
\\For $u_n\to u$ in $H^1$, since the convergence holds also in $L^4$, $F_{u_n}(\theta) \to F_u(\theta)$ for any $\theta$. Hence the sequence $\tilde\Theta(u_n)$, bounded in $H^1$, converges weakly to the minimizer $\tilde\Theta(u)$; again by continuity of the minimal value with respect to $u$, and convergence of the norms, the convergence is actually strong.\\For $u \in H^1_{rad}$ the variational problem $F_u$ can be considered for $\theta \in H^1_{rad}$. It follows naturally $\tilde{\Theta}(u)\in H^1_{rad}$. \\If $u$ is radially decreasing, let $\phi$ be the symmetric rearrangement of $\Theta(u)$; then by Pólya–Szegö inequality and basic properties for symmetric rearrangements (see \cite{Arg} and reference therein) $F_u(\phi)\le F_u (\Theta(u))$.\\Finally, estimates \eqref{0an} follows, by density argument, by the analogous estimate given in Theorem \ref{teoang} for $u  \in L^\infty\cap L^4$.
\end{proof}
We keep saying $\Theta(u) $ in the rest of the paper to refer to the solution $\tilde \Theta(u) $ of this Lemma.\bigskip\\
 We want to lose the dependence on the $ {L^\infty}$ norm on the right hand side of estimate \eqref{stan}. It is enough for our aim to get a bound on $H^1$ norm of the difference $\theta_1- \theta_2$. We closely follow the proof of Theorem \ref{teoang}, with the necessary modifications: we unbalance the inequality at the expense of $u_1$ to avoid the undesired dependence on $\|u_1\|_{L^\infty},\|u_2\|_{L^\infty}$. \\
 In particular, the reasons for a bound of the $L^\infty $ norm in the original proof will be absorbed by technical hypothesis on $u_1$ only. We assume them in the following Lemma, and later we prove they are fulfilled by any $0\neq u_1 \in H^1_{rad}$. This procedure makes the dependence of the constant $C$ on $u_1$ much worse; nonetheless, the estimate is used with a fixed $u_1$ and $u_2$ a small perturbation in $H^1$.
\begin{lemma} \label{lat}
   Let $u_1, u_2\in  H^1(\R^2)$, and $\theta_1, \theta_2$ the respective angles given by Lemma \ref{minan}. Assume also there exist $\varepsilon>0$, $\alpha< \pi/4$ such that the following implication holds:
    \begin{equation}
    \label{tec}
    \mathrm{for\ a.e.\ } x\in \R^2, \ \mathrm{if\ \ }\theta_1(x) > \alpha\ \ \ \ \mathrm{then\ \ } |u_1(x)| \ge \varepsilon
    \end{equation}
    Then exists $C= C\left ( \|u_1\|_{H^1}, \|u_2\|_{H^1}, \varepsilon, \alpha, q ,\lambda\right )$ such that
\begin{equation}
          \label{stan2}
    \|\theta_1 - \theta_2\|_{H_1} \le C \| u_1-u_2\|_{H^1}
\end{equation}
\label{leman1}
\end{lemma}
\begin{remark}
    The hypothesis \eqref{tec} requires that a control from below $\theta_1(x)> \alpha$ implies a control from below for the modulus of $u_1$, i.e. $|u_1(x)|\ge \varepsilon$.
\end{remark}
\begin{proof}
From \eqref{eqcal2}, the difference $\theta_1-\theta_2$ satisfies:
    \begin{equation*}
        -\lambda \Delta (\theta_1-\theta_2) = -q(\sin (2\theta_1) - \sin (2\theta_2)) + 2|u_1|^2 (\cos (2\theta_1) - \cos (2\theta_2)) + 2 (|u_1|^2 - |u_2|^2) \cos(2\theta_2)
    \end{equation*}
    We multiply by $\theta_1-\theta_2$ and integrate in $\R^2$; we can integrate by parts neglecting the boundary terms the left hand side, as $\theta_i \in H^2(\R^2)$ . It follows
    \begin{align}
        \label{stm}
        \int_{\R^2} &|\nabla (\theta_1-\theta_2)|^2 = -\frac{q}{\lambda } \int_{\R^2}(\sin(2\theta_1)-\sin(2\theta_2)(\theta_1-\theta_2) +\\ \notag &+\frac{2}{\lambda} \int_{\R^2} 2|u_1|^2 (\cos (2\theta_1) - \cos (2\theta_2))(\theta_1-\theta_2) + \frac{2}{\lambda} \int_{\R^2} (|u_1|^2 - |u_2|^2) \cos(2\theta_2)(\theta_1-\theta_2)
    \end{align}
As in \cite{Arg}, we want to use the inequality, for $0\le x<y\le \pi/4$
    \begin{equation}
       \sin 2y - \sin 2x\ge 2\cos (2y) (y-x)
       \label{sen}
    \end{equation}
    to reconstruct, from the first integral on the right hand side of \eqref{stm}, the $L^2$ norm of $\theta_1-\theta_2$. We have to pay attention to the areas where $\cos(2y)$ is too close to $0$, that is when $2y$ is close to $\pi/2$, as in these areas the smallness of $\cos(2y)$ is worsening the constant that controls $\|\theta_1-\theta_2\|_{H^1}$.\\
Define the set
\begin{equation*}
    A:= \left \{ \theta_2 \le  \alpha  \right \} \cup \left \{ \theta_1 \le \alpha \right \} 
\end{equation*}
For $\tau>0$ such that $\alpha + \tau< \pi/4$, we have the following inclusion
\begin{align*}
    &A \subset A_1 \cup A_2 \cup A_3; \ \ \ \ \ \ 
 A_1:= \left \{ \theta_2 \le  \alpha + \tau \right \} \cap \left \{ \theta_1 \le \alpha  + \tau\right \} ;  \\& A_2 :=\left \{ \theta_2 \ge  \alpha + \tau \right \} \cap \left \{ \theta_1 \le \alpha \right \}; \ \ \ A_3:=\left \{ \theta_2 \le  \alpha \right \} \cap \left \{ \theta_1 \ge \alpha + \tau\right \} 
\end{align*}
On $A_1$, from \eqref{sen} we have for a positive constant depending on $\alpha, \tau$:
\begin{equation}
    |\sin(2\theta_1)- \sin(2\theta_2)| \ge C_{\alpha, \tau} |\theta_1- \theta_2|
\end{equation}
A similar inequality holds with a possibly different constant on the remaining part of the set $A$. From $0 \le \theta_1, \theta_2\le \pi/4$, on $A_2$ we have $\tau \ge C_\tau  |\theta_2 - \theta_1|$ for a constant $C_\tau$ depending on $\tau$; hence
\begin{equation}
    \sin (2\theta_2) - \sin (2\theta_1) \ge \sin ( 2(\alpha + \tau))- \sin (2\alpha)\ge C_{\alpha, \tau} \tau \ge C_{\alpha, \tau}|\theta_2-\theta_1|
    \label{bco}
\end{equation}
By a symmetric argument, the same holds on $A_3$.\\
As $\sin $ is increasing in $[0,\pi/2]$, we can estimate the first integral on the right hand side of \eqref{stm}:
\begin{equation}
\label{ob}
 -\int_{\R^2}(\sin(2\theta_1)-\sin(2\theta_2)(\theta_1-\theta_2)\le -\int_A (\sin(2\theta_1)-\sin(2\theta_2)(\theta_1-\theta_2) \le - C_\tau C_{\alpha, \tau}\int_A (\theta_1-\theta_2)^2 
\end{equation}
We want to prove a similar estimate for the second integral in \eqref{stm} to obtain the $L^2$ norm of $\theta_1-\theta_2$ on the set $A^C$. The function $\cos$ is decreasing in $(0,\pi/2)$, and therefore the integrand is negative. Similarly to \eqref{sen} we can use 
  \begin{equation*}
       \cos y - \cos x\le \sin (x) (x-y)\le \sin (\alpha) (x-y)  \ \ \ \mathrm{for \ } y>x\ge \alpha
    \end{equation*}
In $A^C$ $\theta_1, \theta_2 \ge \alpha$; from the previous inequality
\begin{equation*}
    (\cos (2\theta_1) - \cos (2\theta_2))(\theta_1-\theta_2) \le - 2\sin (\alpha) (\theta_1-\theta_2)^2\ \ \  \mathrm{in \  } A^C
\end{equation*}
We estimate the second integral on the right hand side of \eqref{stm}
\begin{align}
    \notag  & \int_{\R^2} 2|u_1|^2 (\cos (2\theta_1) - \cos (2\theta_2))(\theta_1-\theta_2)\le  \int_{A^C} 2|u_1|^2 (\cos (2\theta_1) - \cos (2\theta_2))(\theta_1-\theta_2) \\&\le -  C_{\alpha}\int_{A^C} |u_1|^2 (\theta_1-\theta_2)^2 \le -C_{\alpha, \varepsilon} \int_{A^C} |\theta_1- \theta_2|^2
    \label{oo}
\end{align}
In the last inequality we have used the condition \eqref{tec}, to control from below $|u_1|$ on the set $A^C$. Combining \eqref{ob}, \eqref{oo} and \eqref{stm}, we have with $C_1, C_2$ depending on $\alpha, \varepsilon, \tau, q, \lambda$:
\begin{align*}
    \|\nabla (\theta_1 - \theta_2)\|_{L^2(\R^2)}^2 + &C_1\|\theta_1 - \theta_2\|_{L^2(\R^2)}^2 \le \frac{2}{\lambda} \int (|u_1|^2 - |u_2|^2) \cos(2\theta_2)(\theta_1-\theta_2)\\&\le  C_2 \left (\|u_1\|_{L^4\left(\R^2\right)}+\|u_2\|_{L^4\left(\R^2\right )}\right )^2\|u_1-u_2\|_{L^4\left (\R^2\right )}^2+  \frac{C_1}{2}\|\theta_1- \theta_2\|_{L^2\left (\R^2\right )}^2 
\end{align*}
We have applied Holder and weighted Cauchy Schwartz inequalities. At this point the proof follows as in \cite{Arg} applying the Gagliardo-Niremberg inequality for $v\in H^1(\R^2)$
\begin{equation*}
    \|v\|_{L^4(\R^2)}^2 \le C \|\nabla v\|_{L^2(\R^2)}\|v\|_{L^2(\R^2)}
\end{equation*}
\end{proof}
Condition \eqref{tec} is a very weak request, if one has the freedom to choose parameters $\alpha, \varepsilon$ depending on $u$. We have the following:
\begin{lemma}
    Consider $0\neq u\in H^1_{rad}$, and $\theta \in H^1_{rad}$ be the associated angle given by Lemma \ref{minan}. Then there exists $0<\alpha< \pi/4$, $\varepsilon>0$ depending on $u, \theta $ such that the condition \eqref{tec} holds. 
    \label{leman}
\end{lemma}
\begin{proof}
     $u\in H^1_{rad}$ and by Lemma \ref{minan} $ \theta\in H^2_{rad}$; this implies that $u, \theta, \theta':= \frac{d}{dr} \theta$ are continuous in $\R^2\setminus \left \{ 0\right \}$. In the rest of the proof, we identify the functions with their continuous representatives, and we write $u,\theta , \theta'$ both for the functions in $\R^2$ and their radial restriction.  
     \\By contradiction, assume that for any $\alpha < \pi/4, \varepsilon>0$ there exists a non zero measure set
     \begin{equation*}
         A_{\varepsilon, \alpha} := \left \{| u| \le \varepsilon \right \} \cap \left\{ \theta \ge \alpha \right \}
     \end{equation*}
For sequences $\varepsilon_n \downarrow 0, \ \alpha_n \uparrow \pi/4$, $A_n := A_{\varepsilon_n, \alpha_n}$ defines a sequence of radially symmetric decreasing sets with nonzero measures. 
\\Assume there exists $R>0$ such that $\partial B_{R}(0) \subset\bigcap\limits_{n} A_n$. By continuity $u(R)=0; \theta(R)= \pi/4$. Since $\theta\le \pi/4 $ by Lemma \ref{minan}, $R$ is a local maximum for $\theta$; moreover, from the radial expansion of \eqref{eqcal2} we infer $\theta\in C^2$ in a neighbourhood of $R$. But this leads to a contradiction, as evaluating equation \eqref{eqcal2} in $R$
\begin{equation*}
    \lambda \theta_{rr}(R) = q>0
\end{equation*}
implies that $\theta$ is strictly convex in a point of local maximum.
\\If $\bigcap\limits_n A_n$ does not contain a set bounded away from $0$, then it must hold $B_{R_n} (0) \subset A_n$ for any $n$, and for radius $R_n>0$. 
\\We adapt in the following the proof of Corollary 3.1 in \cite{Arg}. For $n$ large fixed, define $\beta_n:= \theta(R_n)$. By definition we have $\alpha_n \le \beta_n \le \pi/4$. 
\\Multiplying equation \eqref{eqcal2} by $(\theta - \beta_n)^+$ and integrating by parts over $B_{R_n}$, we infer
\begin{equation*}
    \lambda\int_{B_{R_n} (0)}\left |\nabla \left (\theta - \beta_n \right ) ^+ \right|^2 = \int_{B_{R_n} (0)} \left(- q \sin (2\theta) + |u^2| \cos (2 \theta) \right ) (\theta - \beta_n)^ + \le 0  
\end{equation*}
with the last inequality holding for $n$ sufficiently large. This implies $\theta \le \beta_n$ a.e. in $B_{R_n}(0)$. 
\\Here we reach the contradiction: if $\beta_n<\pi/4$, it contradicts the definition of the nonzero measure set $A_m$ for $\alpha_m> \beta_n$, $m>n$; if $\beta_n=\pi/4$ we can repeat the same argument as before with $R=R_n$.
\end{proof}
\begin{remark}
\label{das}
Summing up the Lemmas in this section, we have for any $u$ radial
\begin{equation}
      \|\Theta(u) - \Theta(w) \|_{H_1} \le C_u \| u-w\|_{H^1}
      \label{scad}
\end{equation}
for any $w\in H^1$, with the estimate for $u=0$ following from \eqref{0an}. The constant $C$ has a strong dependence on $u$, since it implicitly depends on its shape through Lemma \ref{leman}. \\Recall that by Lemma \ref{minan}, for $u_n$ converging in $H^1$  to $u$ radial  
\begin{equation*}
    \Theta(u_n) \to \Theta(u) \ \ \mathrm{in \ } H^1; \ \ \ \ \lim_n E(u_n, \Theta(u_n)) = E(u,\Theta(u))
\end{equation*}
For both the results, the estimate \eqref{stan} was not strong enough.
\end{remark}
We add a simple result on the relation between frequencies and charge of the ground states, which comes as a consequence of Theorem \ref{min}, and we recall a property of differentiability for monotonous functions.
\begin{definition}
    For $a>a_0$, $a_0$ as in Theorem \ref{min}, we define 
    \begin{equation}
        \sigma (a) := \inf \left \{ \sigma \ | \ \exists (v, \phi) \mathrm{\ ground\  state\ over \ } S_a \mathrm{ \ solution \ of \ } \eqref{sigma}-\eqref{sig} \mathrm{\ for \ }\sigma \right \}
    \end{equation}
\end{definition}
\begin{lemma}
\label{Lema}
    In the definition of $\sigma (a) $ the infimum is attained; moreover $\sigma$ is strictly increasing in $a$.     
\end{lemma}
    \begin{proof}
By the proof of Theorem \ref{min}, a sequence of ground states in $S_a$, up to translation and subsequences, converges to a ground state strongly in $H^1$. For $\sigma_n$ to $\sigma$, the strong limit of a sequence of solutions of \eqref{sigma}-\eqref{sig} for $\sigma_n$ solves the same system for $\sigma$.\\
   It will be useful, here and in the following, the decomposition of the energy $E:=E^+ + E^-$ with 
\begin{equation}
\label{ende}
    E^-(u, \theta):= \frac{1}{4}\int_{\R^2} |\nabla u|^2- 2|u|^2 \sin (2\theta) \,dx; \ \ \ E^+(\theta):=  \frac{1}{4}\int_{\R^2} \lambda |\nabla \theta|^2 + q(1-\cos(2\theta) \,dx
    \end{equation}
   Let $a_1, a_2>a_0$; for $(v_i, \phi_i) \in S_{a_i}$ a ground state we call $E^-(v_i,\phi_i)= E^-(a_i)$. Then we have
\begin{equation}
    J_{a_2} \le E\left (\sqrt{\frac{a_2}{a_1}} v_{a_1}, \phi_{a_1}\right ) = \frac{a_2}{a_1}E^-(a_1) + E^+(\phi_{a_1}) = J_{a_1} + \left ( \frac{a_2}{a_1}-1\right )E^-(a_1)  
    \label{endif}
    \end{equation}
  The inequality also holds exchanging $a_1,a_2$. Applying \eqref{endif} twice we get
\begin{equation*}
    J_{a_1} \le J_{a_2} + \left ( \frac{a_1}{a_2}-1\right ) E^-(a_2) \le J_{a_1} + \left ( \frac{a_1}{a_2}-1\right ) E^-(a_2) +\left ( \frac{a_2}{a_1}-1\right ) E^-(a_1) 
\end{equation*}
Rearranging we have
\begin{equation*}  \frac{a_2-a_1}{a_2} E^-(a_2)=  \left(1-\frac{a_1}{a_2}\right )E^-(a_2) \le  \left(\frac{a_2}{a_1} -1\right ) E^-(a_1)=\frac{a_2-a_1}{a_1}E^-(a_1) 
\end{equation*}
For $a_2>a_1$, we get the result since \eqref{sigma} implies 
   \begin{equation}
     -4   E^-(v, \phi)= 2\sigma \|v\|_{L^2}^2 
     \label{lala}
    \end{equation}
    \end{proof}
    \begin{lemma}
Let $f:I \subset \R \to \R$ be a strictly monotonous function. Then the following quantity is finite for a.e. $x\in I$
\begin{equation*}
    \bar D^+ f(x) := \limsup_{h\to 0^+ } \frac{f(x+h) - f(x)}h 
\end{equation*}
\label{Lemmon}
    \end{lemma}
    The proof is classical, see for example \cite{Mon}. We adopt this notation in the following: for $u\in H^1$ and $y\in \R^2$, $u_y(x):= u(x+y)$. For $Lu:= -\Delta u + V(x) u$ a linear operator defined in $H^2\left ( \R^2; \R\right ) \subset L^2$, we denote its natural extension as a bilinear form in $H^1\left ( \R^2; \R\right ) \times H^1\left ( \R^2; \R\right ) $ as
\begin{equation*}
    \left \langle L u , v \right \rangle := \int \nabla v \nabla u + \int V(x) uv
\end{equation*}
while for $X$ an Hilbert space we define $(u, v)_X$ the scalar product between $u, v$ in $X$.
\section{Stability of ground states}\label{secstab}
\subsection{Orbital Stability}
We begin defining the action 
\begin{equation}
    S(u, \theta) := E(u, \theta) + \frac{\sigma}{2} \int_{R^2} |u|^2\,dx= E(u,\theta) + \sigma Q(u)
\end{equation}
with $\sigma$ given by equation \eqref{sigma}, which we consider fixed throughout this section.

 For $(v,\phi)$ a ground state verifying equations \eqref{sigma}-\eqref{sig}, we have $S'(v,\phi) \equiv0 $. Recalling that $v$ is real valued, we can compute the second derivative of $S$ at $(v,\phi)$, denoted as $\left. S^{''}\right.|_{v,\phi}$: 
\begin{equation}
    \label{calc''}
    \left \langle \left. S^{''}\right.|_{v,\phi}(\eta, \theta)^t,(\eta,\theta)^t\right \rangle = \frac{1}{2} \left \langle \mathcal{L}_1^{(v,\phi)}(\eta_1, \theta)^t,(\eta_1,\theta)^t\right \rangle +  \frac{1}2\left \langle \mathcal{L}_2^{(v,\phi)}\eta_2,\eta_2\right \rangle 
\end{equation}
Here we have decomposed $\eta= \eta_1+ i \eta_2 $ with $\eta_1, \eta_2$ real valued; $v^t$ indicates the transpose of the vector $v$, and the operators are defined as 
\begin{align}
 &\mathcal{L}_1^{(v,\phi)}:= 
    \begin{pmatrix}
  L_1^{(v,\phi)} & -4v\cos (2\phi) \\ -4v\cos (2\phi) & L_2^{(v,\phi)} \end{pmatrix};   \ \ \ \ \mathcal{L}_2^{(v,\phi)} = L_1^{(v,\phi)};  \
  \\ & L_1^{(v,,\phi)}:= -\Delta +2\sigma - 2\sin (2\phi); \ \ \ \ L_2^{(v,\phi)} := -\lambda \Delta + 2q \cos (2\phi) + 4|v|^2\sin (2\phi) 
\label{defop}
\end{align}
When the ground state is fixed with no possibility of confusion, we occasionally drop the superscript 
$(v,\phi)$. 
For a control over the spectra of $\mathcal{L}^{(v,\phi)}$ that is uniform for $(v,\phi)\in \mathcal{M}^\sigma$, we need the following compactness result.
\begin{lemma} There exists $a>0 $ such that $\mathcal{M}^\sigma \subset S_a$, and any $(w,\psi) \in \mathcal{M}^\sigma$ is a ground state over $S_a$. Moreover, $\mathcal{M}^\sigma $ is compact in the $H^2\times H^2$ topology.
\end{lemma}
\begin{proof}
    Let $(v,\phi) \in \mathcal{M}^\sigma $ be a ground state over $S_{\|v\|_{L^2}^2}$. If $(w, \psi ) \in \mathcal{M}^\sigma$ satisfies $\lambda^2 \|w\|_{L^2}^2 = \|v\|_{L^2}^2$ for $\lambda \neq 1$, then we have the contradiction by Lemma \ref{minan}
    \begin{equation*}
        S(\lambda w , \Theta( \lambda w) ) < S(\lambda  w, \psi) = S(w , \psi ) = S(v, \phi)
    \end{equation*}
    using in the first equality $(w, \psi)$ solution of \eqref{sigma}-\eqref{sig}.
\\Any sequence in $\mathcal{M}^\sigma$ is then a minimizing radial sequence for the energy in $S_a$. Strong convergence up to subsequences in $H^1$ to a minimal configuration was already proved in Theorem \ref{min} (see \cite{Arg}), and the limit satisfies equations \eqref{sigma}-\eqref{sig} with the same $\sigma$. Finally, convergence in $H^2$ is deduced by elliptic regularity argument for \eqref{sigma}-\eqref{sig}. 
\end{proof}
The main result of this section is the following
\begin{proposition}
     Let $a>a_0$, $a_0$ as in Theorem \ref{min}. If the map $\sigma(a)$ verifies $\bar{D}^+ \sigma(a) < \infty$, there exists $\tau >0 $ such that 
    \begin{equation}
    \label{''} 
    \left \langle  S^{''}|_{v,\phi}(\eta, \theta)^t, (\eta, \theta)^t\right \rangle \ge \tau\left  ( \| \eta \|_{H^1}^2+ \| \theta \|_{H^1}^2\right ) 
\end{equation}
for any $(v,\phi) \in \mathcal{M}^{\sigma(a)}$,  $\eta, \theta \in H^1 $ with norm sufficiently small and
\begin{align}
\label{cond6}
\|\eta + v\|_{L^2}=&\|v\|_{L^2} \\ \ \ 
0 = (\eta, iv)_{H^1} &= ((\eta,\theta)^t, (\partial_l v, \partial_l \phi)^t)_{H^1\times H^1}= ((\eta,\theta)^t, (k_m, h_m)^t)_{H^1\times H^1}
\label{cond3}
\end{align}
where $(k_m, h_m)$ is an orthonormal basis of $\mathrm{Ker}\mathcal{L}_1^{(v,\phi)} \cap \left ( H^1_{rad} \times H^1_{rad} \right )$.
\label{S''}
\end{proposition}
   \begin{remark}The first orthogonality in \eqref{cond3} are due to invariance by translations and multiplication by a complex exponential of the action. As was noted in \cite{Fuk}, \cite{StS}, the invariance implies that the inequality \eqref{''} cannot hold along those directions. Differentiating with respect to $\alpha$ the identity $S'(e^ {i\alpha} v, \phi)\equiv 0$ we infer $\left. S^{''}\right.|_{v,\phi} (iv,0) \equiv 0$, and the same applies for translation invariance.
   \\The last orthogonality is related to modulation in the manifold $\mathcal{M}^\sigma$. This maintains an implicit dependence on $(v,\phi)$ as the basis $(k_m, h_m)$, and even its cardinality, are determined by the specific ground state. To simplify the notation, we omit the indexing specifications of $m$; we will always be considering the number corresponding to the cardinality of the basis. Notice $\mathcal{L}_1^{(v,\phi)}$ is a compact perturbation of a self adjoint operator with positive essential spectrum, so by Weyl's Theorem the Kernel has finite dimension for any $(v,\phi)\in \mathcal{M}^\sigma$. 
    \label{0der}
     \end{remark}
     \begin{remark}
         If the map $\sigma(a)$  had a $C^1$ inverse, the hypothesis on $\bar D ^+ \sigma (a)$ would be equivalent to \eqref{dp}. We have not such regularity on $\sigma$: in particular proving the bijection between frequencies and charge seems quite challenging. A difficulty of our proof is to rely on this weaker notion of derivative.
     \end{remark}
Assuming Proposition \ref{S''}, we prove Theorem \ref{teost}. Without risk of confusion, we occasionally drop the subscript of the norm $\| \cdot \|_{H^1 \times H^1}$.
\begin{proposition}
  Let $\mathcal{M}^\sigma \subset S_a$ satisfy the hypothesis of Proposition \ref{S''}. There exist constants $D>0, \delta>0$ such that for any $u, \varphi \in H^1$ with 
\begin{equation}\|u\|_{L^2}^2=a; \ \ \ \ \inf_{\substack  {0 \le \beta\le 2\pi, \ y \in \R^2  \\ (w, \psi) \in \mathcal{M}^\sigma }} \left \|(u, \varphi) -(e^{i\beta}w_y, \psi_y)\right \|_{H^1\times H^1}= \left \|(u, \varphi) -(e^{i\alpha}v_x, \phi_x)\right \|\le \delta 
\label{conf}
\end{equation}
it holds
\begin{equation}
    E(u,\varphi)  - E(v, \phi) \ge D \inf_{\substack  {\beta, y  \\ (w, \psi) \in \mathcal{M}^\sigma }} \left \|(u, \varphi) -(e^{i\beta}w_y, \psi_y)\right \|_{H^1\times H^1}^2
\end{equation}
\label{proco}
\end{proposition}
\begin{remark}
The infimum problem that defines the modulation parameters in the hypothesis \eqref{conf} is attained. $\|u-e^{i\beta}w_y\|_{H^1}$ is a continuous function of $(w,\psi), \beta, y\in \mathcal{M}^\sigma \times[0,2\pi]\times \R^2$; the first two sets are compact, and it is straightforward to verify that for a minimizing sequence $y_n$ is bounded.
\end{remark}
\begin{proof}
    We follow closely the proof in \cite{Fuk}. Consider $\alpha \in \R$, $x\in \R^2, (v, \phi) \in \mathcal{M}^\sigma$ such that
\begin{equation}
\inf_{\substack  {\beta, y  \\ (w, \psi) \in \mathcal{M}^\sigma }} \left \|(u, \varphi) -(e^{i\beta}w_y, \psi_y)\right \|_{H^1\times H^1}= \left \|(u, \varphi) -(e^{i\alpha}v_x, \phi_x)\right \|_{H^1\times H^1}\le \delta
\end{equation}
with $\delta$ to be fixed. $\eta:=e^{-i\alpha} u_{-x} - v$ and $\theta:= \varphi_{-x} -\phi$ satisfy the hypothesis of Proposition \ref{S''}: \eqref{cond6} by hypothesis \eqref{conf}, and \eqref{cond3} by $x, \alpha, (w,\psi)$ minimizing the distance over the symmetry group and the variety $\mathcal{M}^\sigma$. A Taylor expansion and Proposition \ref{S''} lead to
\begin{align}
    \notag S(u,\varphi)& = S(e^{-i\alpha} u_{-x}, \phi + \theta) \\& \notag = S(v, \phi) +  S'(v,\phi ) (\eta, \theta) + \frac{1}{2} \left \langle \left. S^{''}\right.|_{v,\phi}(\eta, \theta)^t, (\eta, \theta)^t\right \rangle + o\left (\|\eta\|_{ H^1}^2 + \|\theta \|_{ H^1}^2\right )
    \\ & \ge  S(v, \phi) + \tau \left (\|\eta\|_{H^1}^2 +\|\theta\|_{H^1}^2\right )  +  o\left (\|\eta\|_{ H^1}^2 + \|\theta \|_{ H^1}^2\right )
    \label{30}
\end{align}
The first order term vanishes for $(v,\phi)$ solution of the system \eqref{sigma}-\eqref{sig}. 
  The $L^2$ norms of $v$ and $u$ are equal, so the conclusion follows directly from the previous inequality for $\delta$ sufficiently small.
\end{proof}
We conclude the prove of Theorem \ref{teost} mimicking Theorem 3.5 in \cite{StS}. 

\begin{proof}[Proof of Theorem \ref{teost}] By Lemmas \ref{Lemmon}-\ref{Lema}, for a.e. $a>a_0$ $\sigma(a)$ satisfies the derivative hypothesis of Proposition \ref{''}.\\
    Assume the thesis were false for $a$ as above, and $\sigma= \sigma(a)$ fixed. We have, for $\delta$ given by Proposition \ref{proco}, a value $0<\varepsilon<\delta$, sequence of initial datum $u_{n}^0$ and a sequence of times $t_n$ and $(v,\phi) \in \mathcal{M}^\sigma $ such that
    \begin{equation}
        \inf_{\substack  {\beta, y  \\ (w, \psi) \in \mathcal{M}^\sigma }} \left \|(u_n, \theta_n)(t_n) -(e^{i\beta}w_y, \psi_y)\right \|_{H^1\times H^1}\ge \varepsilon; \ \ \ \ u_n^0 \to v
    \end{equation}
    For Theorem \ref{teoca}, Schrödinger-Poisson flow is continuous in time; hence we can pick sequences $t_n\in \R, \alpha_n \in [0,2\pi], x_n \in \R^2,$ $(v_n, \phi_n) \in \mathcal{M}^\sigma$ such that
 \begin{equation}\label{ora}
    \inf_{\substack  {\beta, y  \\ (w, \psi) \in \mathcal{M}^\sigma }} \left \|(u_n, \theta_n)(t_n) -(e^{i\beta}w_y, \psi_y)\right \|
    = \left \|(u_n, \theta_n)(t_n) -(e^{i\alpha_n}(v_n)_{x_n}, (\phi_n)_{x_n})\right \|
    =  \varepsilon
    \end{equation}
    Set $u_n:= e^{-i\alpha_n}(u_n)_{-x_n}(t_n)$. By Lemma \ref{minan} $\Theta\left (u_0^n\right ) \to \phi$ strongly in $H^1$; conservation of energy and charge implies
   \begin{equation*}
E(v,\phi)=  \lim_n E\left ( u_n^0, \Theta\left (u_n^0\right )\right )  = \lim_n E\left ( u_n, \Theta(u_n ) \right ); \ \ \|v\|_{L^2}= \lim_n \|u_n^0\|_{L^2} = \lim_n \|u_n\|_{L^2}  
   \end{equation*}
   We define $w_n:= u_n\frac{ \|v_n\|_{L^2}}{\|u_n\|_{L^2}}$, which has the same $L^2$ norm of the ground state in $\mathcal{M}^\sigma$. Since $\|w_n- u_n\|_{H^1} \to 0$ for $n\to \infty$, for $n$ sufficiently large we have
   \begin{equation}
   \label{lok}
        \frac{1}{2} \left \|u_n-v_n\right \|_{H^1}
        \le \left \|w_n -v_n\right \|_{H^1}
        \le 2 \left \|u_n-v_n,\right \|
        _{H^1}
   \end{equation}
By \eqref{ora} the couple $(w_n, \Theta(u_n))$ satisfies the hypothesis \eqref{conf} of Proposition \ref{proco}. We have reached the contradiction as we would have
    \begin{align*}
        0= \lim_n E(w_n, \Theta (u_n)) - E(v,\phi) \ge D  \limsup_n \inf_{\substack {\beta,y\\(v,\phi) \in \mathcal{M^\sigma}} }\|(w_n, \Theta( u_n) )  - (e^{i\alpha}v_y, \phi_y) \|
        \ge D \varepsilon 
    \end{align*}
\end{proof}
The remaining part of this section is dedicated to the proof of Proposition \ref{S''}, which we split in several Lemmas. 
\\Our argument is based on a study of the spectrum of the operators $\mathcal{L}_1, \mathcal{L}_2$. An orthogonality condition fits better such techniques, rather than the norm equality \eqref{cond3}. For this we have the following
\begin{lemma}
\label{prim}
    Suppose there exists $\tau'>0$ such that the inequality \eqref{''} holds for any $(z, \psi) \in  H^1\times H^1$ satisfying \eqref{cond3} and
    \begin{equation}
        \label{cond4}(z, v)_{L^2}= 
0
    \end{equation}
    Then exist $\tau>0$ such that \eqref{''}  holds for any $(\eta, \theta) \in  H^1\times H^1$ satisfying \eqref{cond6}-\eqref{cond3}, with norm sufficiently small. 
\end{lemma}
\begin{proof}
Consider $(\eta, \theta)$ satisfying \eqref{cond6}-\eqref{cond3}, and decompose the perturbation as
\begin{align*}
  \begin{pmatrix}
      \eta \\ \theta
  \end{pmatrix}  = t\begin{pmatrix}
      v \\ \varphi
  \end{pmatrix}  + \begin{pmatrix}
      z \\ \psi
  \end{pmatrix} ; \ \ \ t\in \mathbb{C}; \ \ \ z, \varphi,  \psi \in H^1(\R^2); \ \ \ (z,v)_{L^2}=0
\end{align*}
Here $\varphi$ is a radial, real valued function. Since
$(v, \partial_l v)_{H^1}= (v, iv)_{H^1}=0$; $(z, \psi)$ satisfies the first two orthogonality condition in \eqref{cond3}. We can fix $\varphi$ such that  
\begin{equation*}
    ((v,\varphi)^t, (k_m, h_m)^t)_{H^1\times H^1} = 0
\end{equation*}It follows that $(z, \psi)$ satisfies conditions \eqref{cond3}. \\
The Taylor expansion of the charge gives 
\begin{equation*}
    Q(v+ \eta)= Q(v) + Q'(v)(\eta) + O (\|\eta \|_{L^2}^2)= Q(v) + t Q'(v)(v) + O (\|\eta \|_{L^2}^2)
\end{equation*}
Since $Q(v+\eta)= Q(v)$ by hypothesis, and $Q'(v)(v)$ is a fixed positive quantity, independent on $\eta$, we get $t=O(\|\eta \|_{L^2}^2)$. \\
Hence, both in $L^2$ and in $H^1$ we have
\begin{align*}
&\|z\|_{L^2}^2= \|\eta\|_{L^2}^2 +t^2 \|v\|_{L^2}^2 -2 t (\eta, v)_{L^2} = \|\eta\|_{L^2}^2 + o(\|\eta\|_{L^2}^2)
\\&\|z\|_{H^1}^2= \|\eta\|_{H^1}^2 +t^2 \|v\|_{H^1}^2 -2 t (\eta, v)_{H^1} = \|\eta\|_{H^1}^2 + o(\|\eta\|_{H^1}^2)
\end{align*}
The first line follows from in the relation above, while in the second  we have used $o(\|\eta\|_{L^2}^2) / \|\eta\|_{H^1}^2 \to 0 $ as $\eta \to 0 $ in $H^1$. 
%
It follows then
\begin{align*}
    &\left  \langle \left. S^{''}\right.|_{v,\phi}(\eta, \theta)^t, (\eta, \theta)^t\right  \rangle= \left \langle \left. S^{''}\right.|_{v,\phi}(z, \psi)^t, (z, \psi)^t\right \rangle + o\left (\|\eta\|_{H^1}^2+\|\theta\|_{H^1}^2\right )\\ &\ge \tau' \left (\|z\|_{H^1}^2 +\|\psi\|_{H^1}^2\right )  +  o\left (\|\eta\|_{H^1}^2+\|\theta\|_{H^1}^2\right ) \ge \tau \left (\|\eta\|_{H^1}^2 +\|\theta\|_{H^1}^2\right ) + o\left (\|\eta\|_{H^1}^2 +\|\theta\|_{H^1}^2\right) 
\end{align*}
\end{proof}
We study the operators $\mathcal{L}_1, \mathcal{L}_2$ separately. For $\mathcal{L}_2$, the proof is a direct application of a standard result in spectral analysis. For convenience of the reader we report an adaptation of the proof from \cite{Arghir}.
\begin{lemma}
\label{Spec}
    Let $\phi \in H^1$, $L_\phi(u) := (-\Delta+ 2\sigma - 2\sin (2\phi)   )u$ . If there exists $v \in H^2$ satisfying $v>0$; $L_\phi v=0$, then
    \begin{equation}
    \left \langle L_\phi u, u \right \rangle \ge 0 \ \ \ \ \forall \ u\in H^1(\R^2); \ \ \ \ \ \\ \label{maggio1} \left \langle L_\phi u, u \right \rangle > 0 \ \ \ \ \forall \ u\in H^1(\R^2)\setminus  \R v
    \end{equation}
\end{lemma}
\begin{proof}
    For $u \in C^\infty_C$, we define $g$ as $g:= \frac u v $. We can compute, by integration by parts and hypothesis on $v$: 
    \begin{align*}
        \left \langle L_{\phi}u, u\right \rangle&= -\int g^2 v \Delta v  -\int v^2 g \Delta g - 2\int vg \nabla g \nabla v + 2\sigma \int g^2 v^2 - 2\int g^2\sin (2\phi) v^2  
        \\ &= \int g^2 vL_{\phi } v  -\int v^2 g\Delta g - \int g \nabla g  \nabla (v^2)\\&= 
        \int v^2 | \nabla g|^2
    \end{align*}
    By standard density argument, the latter equality holds for any $u \in H^1$. Hence we have $\left \langle Lu , u\right \rangle\ge 0 $ for any $u$, and the equality holds if and only if $\nabla g= 0$, i.e. $u= cv$ for a constant $c$. 
\end{proof}
\begin{lemma}
    Then there exists $\tau_2>0 $ such that
    \begin{equation*}
        \left \langle \mathcal{L}_2^{(v,\phi)} \eta, \eta\right \rangle \ge \tau_2 \|\eta\|_{H^1}^2
    \end{equation*}
    for any $(v,\phi)  \in \mathcal{M}^\sigma$, and for any $\eta\in H^1 \left (\R^2; \R\right)$ satisfying $\left (\eta, v\right )_{H^1}=0$
\end{lemma}
\begin{proof}
$\mathcal{L}_2^{(v,\phi)}= L_{\phi}$, and $v$ is a strictly positive solution of $L_{\phi} v=0 $.
    \\Assume the thesis were false. Since $\mathcal{L}_2 $ is a nonnegative operator by Lemma \ref{Spec}, there exist sequences $(v_j, \phi_j) \in \mathcal{M}^\sigma$, $\eta_j\in H^1$ such that
\begin{equation*}
    0= \lim_{j\to \infty} \left \langle L_{\phi_j}\eta_j, \eta_j \right \rangle; \ \ \ \|\eta_j\|_{H^1}=1; \ \ \ \left (\eta_j, v_j \right ) _{H^1}=0
\end{equation*}
$(v_j, \phi_j) \to (v,\phi) \in \mathcal{M}^\sigma$ strongly in $H^2$; up to subsequences, $\eta_j$ converges to $\eta$ weakly in $H^1$ and in $L^p$ for any $2\le p < \infty$. Product of strong and weak convergence preserves the orthogonality relation and
\begin{equation}
\label{ggg}
    \lim_{j\to \infty }\int |\eta_j|^2\sin (2\phi_j) = \int |\eta|^2 \sin (2\phi)
\end{equation}
It follows
\begin{equation*}
0= \lim_{j\to \infty} \left \langle L_{\phi_j}\eta_j, \eta_j \right \rangle \ge \min \left \{ 1, 2\sigma \right \} - \lim_{j} \int |\eta_j|^2\sin (2\phi_j) = \min \left \{ 1, 2\sigma \right \} - \int |\eta|^2\sin (2\phi)
\end{equation*}
Hence $\eta\neq 0$ and, by weak convergence and equality \eqref{ggg} we have $\left \langle L_\phi\eta, \eta\right \rangle \le 0$ with $(\eta, v)_{H^1}=0;$ but this contradicts Lemma \ref{Spec}.
\end{proof}
We want a similar bound for $\mathcal{L}_1$, following the same steps. In particular we want to prove \begin{enumerate}
    \item  $\mathcal{L}_1^{(v,\phi)}$ is positive in the subspace $L^2$ orthogonal to $v$
    \item The Kernel of $\mathcal{L}_1^{(v,\phi)}$ is spanned by the translation generators and the tangent space to $\mathcal{M}^\sigma$
    \item The spectrum of $\mathcal{L}_1^{(v,\phi)}$ restricted to \eqref{cond3}-\eqref{cond4} is bounded away from $0$ uniformly for $(v,\phi) \in \mathcal{M}^\sigma$
\end{enumerate} 
Since the proof for $\mathcal{L}_1$ is more complicated, we split each point in a different Lemma.
\begin{lemma}
\label{para}
     Let $(v,\phi)$ be a ground state as in Theorem \ref{min}. Then 
     \begin{equation*}
         \left \langle \left. S^{''}\right.|_{v,\phi} (\eta, \theta)^t,(\eta, \theta)^t\right \rangle \ge 0 \ \ \ \ \forall \ \eta, \ \theta \in H^1\ \mathrm{such \ that}\ \left (\eta, v\right )_{L^2}=0
     \end{equation*}
\end{lemma}
\begin{proof}
By equations \eqref{sigma}-\eqref{sig} we have 
\begin{equation*}
    E'(v, \phi)(w, \theta ) = -\sigma Q'(v)(w)
\end{equation*}
For $\alpha:= \phi + \theta$, by the orthogonality relation in $L^2$ we can expand the energy as
\begin{align}
   \notag E(v+\eta, \alpha) &= E(v,\phi) +  E'(v, \phi)(\eta, \theta )+\frac{1}{2} \left \langle E''(v, \phi) (\eta, \theta)^t,(\eta,\theta)^t\right \rangle + o(\|\eta\|_{H^1}^2 +\|\theta\|_{H^1}^2) \\&= E(v,\phi) + \frac{1}2\left \langle  \left.E''\right.|_{(v, \phi)} (\eta, \theta)^t,(\eta,\theta)^t \right \rangle + o(\|\eta\|_{H^1}^2 + \|\theta\|_{H^1}^2) 
    \label{asd}
\end{align} 
To complete the energy comparison, the left hand side must be reduced to an element in $S_a$. We define hence 
\begin{equation*}
    \tilde{u} := \frac{\|v\|_{L^2} }{\sqrt{\|v\|_{L^2}^2 + \|\eta\|_{L^2}^2}} (v+\eta) = k (v+\eta)
\end{equation*}
For the orthogonality condition, one immediately checks $\tilde{u}\in S_a$. By definition of $E^-$ we have
\begin{equation}
    \label{sab}
E^+( \alpha) + E^- (v+\eta,\alpha) = E(\tilde{u}, \alpha) + \left (\frac{1}{k^2}-1\right ) E^-(\tilde{u}, \alpha) \ge E(v,\phi) + \frac{\|\eta\|_{L^2}^2} {\|v\|_{L^2}^2}E^-(\tilde{u}, \alpha)  
\end{equation}
where in the last inequality we have taken advantage of the minimality of $(v,\phi)$. Combining \eqref{asd} and \eqref{sab} we have
\begin{equation}
\label{pegg}
    \left \langle \left. E''\right . |_{v,\phi} (\eta,\theta), (\eta,\theta) \right \rangle \ge  \frac{2E^-(\tilde{u}, \alpha) }{\|v\|_{L^2}^2}\|\eta\|_{L^2}^2 + o(\|\eta\|_{H^1}^2+ \|\theta\|_{H^1}^2) 
\end{equation}
We can infer a positive condition on $\sigma$: in particular multiplying \eqref{sigma} by $v$ and integrating we get
\begin{align}
    &2 \sigma = \frac{\int_{\R^2} -|\nabla v|^2 + 2|v|^2 \sin (2\theta) \,dx}{\|v\|_{L^2}^2}= \frac{-4E^-(v,\phi)}{\|v\|_{L^2}^2} > \frac{- 4J_a}{\|v\|_{L^2}^2} \ge 0 
    \label{peg}
    \\&\sigma < \frac{\int_{\R^2} |v|^2 \sin (2\theta) \,dx}{\|v\|_{L^2}^2}\le 1
    \notag
\end{align}
We have $E^-(\tilde{u}, \alpha)=E^-(v,\phi)+ o(1)$, with $o(1) \to 0$ for $(\|\eta\|_{H^1} + \|\theta\|_{H^1}) \to 0$. Hence combining \eqref{pegg}, \eqref{peg} and the definition of $S$ we deduce for $\eta$ small enough
\begin{align*}
     &\left \langle \left. S^{''}\right.|_{v,\phi} (\eta,\theta), (\eta,\theta) \right \rangle  =\left \langle \left. E''\right . |_{v,\phi} (\eta,\theta), (\eta,\theta) \right \rangle +\sigma \|\eta\|_{L^2}^2 \
     \\
     & \ge 2\frac{  E^-(\tilde{u},\alpha)-E^-(v,\phi)}{\|v\|_{L^2}^2} \|\eta\|_{L^2}^2 + o\left (\|\eta\|_{H^1}^2+\|\theta\|_{H^1}^2\right) = o\left(\|\eta\|_{H^1}^2+ \|\theta\|_{H^1}^2\right)
\end{align*}
\end{proof}
\begin{remark}
\label{para2}
    This naturally implies 
    \begin{equation*}
       \left \langle\mathcal{L}_1^{(v,\phi)}(\eta,\theta)^t , (\eta,\theta)^t\right \rangle  \ge 0 \ \ \ \forall \ (v,\phi )\in \mathcal{M}^{\sigma}; \ \ \eta, \theta \in H^1 \ | \ \left ( \eta , v\right )_{L^2}=0
    \end{equation*}
    by restricting the previous calculation to real valued $\eta$.
\end{remark}
We now characterize the Kernel of $\mathcal{L}_1$. This is an adaptation of Appendix A in \cite{W}.
\begin{lemma}
\label{LemKer}
    Let $(v,\phi)$ be a ground state, $v,\phi \in \mathcal{M}^\sigma$, and $\mathcal{L}_1^{(v,\phi)}$ be as in \eqref{defop}. Then the Kernel of $\mathcal{L}_1^{(v,\phi)} $ is spanned by $(\partial_l v, \partial_l \phi)$, $l=1,2$ and the tangent space to $\mathcal{M}^\sigma$ in $(v,\phi)$.
\end{lemma}
\begin{proof}
Since $v, \phi $, acting as potentials for $\mathcal{L}_1^{(v,\phi)}$, are radial, a $0$ eigenfunction $(u, \theta ) $ is written as the product of a radial function and a spherical harmonic; that is $(u(x), \theta(x))= (f(r) Y_k (\alpha), g(r) Y_k(\alpha))$ with $f, g \in H^1((0,\infty); rdr)$ and $Y_k \in H^1(S^1)$ satisfying
\begin{equation} 
    \begin{cases}
    \label{vk}
        &-\left.\Delta\right|_{S^1} Y_k = \lambda_k Y_k  
        \\ & A_k^1 (f,g) : = \left (-\frac{d^2}{dr^2} - \frac{1}{r} \frac{  d}{dr} + 2\sigma - 2\sin (2\phi) + \frac {\lambda_k }{r^2}\right ) f - 4v\cos(2\phi) g = 0 %
        \\& A_k^2 (f,g) : = \left (-\lambda \frac{d^2}{dr^2} - \lambda \frac{1}{r} \frac{  d}{dr} + q\cos(2\phi) + |v|^2\sin (2\phi) + \lambda \frac {\lambda_k} {r^2}\right ) g - 4v \cos(2\phi)f = 0 
    \end{cases}
\end{equation}
where $k=0,1,2...$, $\lambda_k = k^2 $ are the eigenvalues of the Laplacian on $S^1$ and $Y_k$ is the associated eigenfunction. Such eigenfunctions are the restriction to $S^1$ of harmonic polynomials of degree $k$ (\cite{PSF}, Chapter 7). Moreover $f,g$ satisfy the proper boundary condition in $0$: for $k\ge1 $ $f(0)=g(0)=0$. 
\\We split the proof for different values of $k$
\bigskip\\  $\mathbf{k=1}$  We define $v', \phi'$ by the relation $(\partial_l v, \partial_l \phi)= (v', \phi') \frac{x_l}{|x|}$. \\$(v', \phi')^t\frac{x_l}{|x|}$ belongs to the Kernel of $\mathcal{L}_1^{(v,\phi)}$ by Remark \ref{0der}. Moreover $ v', \phi'$ are strictly negative in $(0,\infty)$ (\cite{GNM}, Theorem 3'), and their angular part is an harmonic polynomial of degree one, so they are a solution of \eqref{vk} for $k=1$.  
\\Adapting Lemma \ref{Spec}, we prove that those are the unique solutions. For $u(r), \theta(r) \in C^\infty_c(0,\infty)$ define $h(r):= \frac{u}{v'}; \ j(r):= \frac \theta{\phi'}$. For $Y_1(\alpha)$ a spherical harmonic of degree one, let $v'(r)Y_1(\alpha) =: V; \ \phi'(r) Y_1(\alpha) =: \Phi$. Hence we get
\begin{align*}
    &\left \langle \mathcal{L}_1 (u Y_1 , \theta Y_1 )^t,  (u Y_1, \theta Y_1)^t\right \rangle= \\&=- \int h^2 V \Delta (V) +h\nabla h \nabla \left(|V |^2\right ) + |V|^2 h \Delta h + \int h^2 |V|^2 (2\sigma- 2\sin(2\phi))- 8\int hj v\cos(2\phi) V \Phi -\\&- \lambda \int j^2 \Phi \Delta (\Phi) +j\nabla j \nabla \left(|\Phi|^2\right )+ |\Phi|^2 j \Delta j  + q \int j^2 (q\cos(2\phi) + |v|^2 \sin (2\phi)) |\Phi|^2 = 
    \\& = \int h^2 V L_1 V + \int j^2\Phi L_2\Phi+ \int V^2 |\nabla h|^2 + \lambda\Phi^2|\nabla j|^2  - 8\int hj v\cos(2\phi) V \Phi=\\& =\int h^2 V \left (L_1V-4\cos (2\phi)v \Phi\right ) + j^2\Phi\left( L_2\Phi  - 4\cos(2\phi) v V \right) + \\& +\int V^2 |\nabla h|^2 + \Phi^2|\nabla j|^2  + 2\int (h-j)^2 v\cos(2\phi) V \Phi  =\\& =   \int V^2 |\nabla h|^2 + \Phi^2|\nabla j|^2  + 2\int (h-j)^2 v\cos(2\phi) V \Phi 
    \end{align*}
    In the second equality we have used integration by parts, in the same spirit as Lemma \ref{Spec}. In the third we have splitted the last integral using the equality 
    \begin{equation*}
        -2hj = -h^2-j^2 + (h-j)^2
    \end{equation*}
    and in the last one we have used $\mathcal{L}_1 (V,\Phi)^t= (0,0)^t$.\\
    Hence by standard density argument
\begin{equation*}
    \left \langle \mathcal{L}_1 (u Y_1 , \theta Y_1 )^t,  (u Y_1, \theta Y_1)^t\right \rangle \ge 0 \ \ \ \forall \ u,\theta \in H^1_0 ((0,\infty) ; rdr)
\end{equation*}
and the inequality is strict unless $(u, \theta) \in \R (v', \phi')$.\bigskip\\ $\mathbf{k\ge2}$ We have, for $c_k^i>0$ 
\begin{equation*}
    A^i_k= A^i_1 + \frac{c_k^i}{r^2}
\end{equation*}
Since $A_1:=\left(A_1^1, A^2_1\right)$ defines a positive quadratic form by the previous calculation, $A_k$ is strictly positive for $k\ge 2$ and does not admit $0$ eigenfunctions.\bigskip\\$\mathbf{k=0}$ Let $M:= \mathrm{Ker} \mathcal{L}_1 \cap \left ( H^1_{rad} \times H^1_{rad} \right )$. By regularity argument $M\subset \left ( H^2 \times H^2 \right )$. \\We can define the $C^1$ function 
\begin{equation}
    F:\left (H^2_{rad}\right )^2 \to \left(L^2\right ) ^2; \ \ \ F(u,\theta) := \begin{pmatrix}
        &-\Delta u + 2\sigma u - 2\sin (2\theta) u
        \\& - \lambda \Delta \theta + q \sin (2\theta) - 2|u|^2\cos( 2\theta)
    \end{pmatrix}
\end{equation}
We have $F(v,\phi)=0$, and the Jacobian of $F$ at $(v,\phi)$ is
\begin{equation*}
     D\left.F\right|_{(v,\phi) } ( u, \theta)^t = \mathcal{L}_1 (u,\theta)^t \ \ \ \forall \ u, \theta \in H^2_{rad}
\end{equation*}
We split the Hilbert space $H^2_{rad} = M \oplus M^{\perp}$. Since $\mathcal{L}_1$ has no zero eigenvalue in $M^\perp$, its restriction is a isomorphism on the image. Hence, by implicit function Theorem, there exist a smooth manifold of solutions of $F(u,\theta)= 0$ in a neighbourhood of $(v,\phi)$, and $M$ is the tangent space to the manifold in $(v,\phi)$\\Since $F(u, \theta)= 0 $ implies $S'(u, \theta)=0 $, over the manifold the action $S$ is constant; the manifold locally coincides with $\mathcal{M }^\sigma$.
\end{proof}
The implicit function argument implies the following structure properties for $\mathcal{M   } ^\sigma$
\begin{corollary}
\label{corM}
    The manifold $\mathcal{M}^\sigma$ is at most of real dimension $2$. Moreover, if $(v_j, \phi_ j) \in \mathcal{M}^\sigma$ converges in $H^1$ to $(v,\phi)$, an orthonormal basis of $\mathrm{Ker}\mathcal{L}_1^{(v_j, \phi_j)} \cap \left ( H^1_{rad} \times H^1_{rad} \right )$ converges to a orthonormal basis of $\mathrm{Ker} \mathcal{L}_1^{(v, \phi)} \cap \left ( H^1_{rad} \times H^1_{rad} \right )$ 
\end{corollary}
\begin{proof}
   A radial integrable solution of $\mathcal{L}^{(v,\phi)}_1 (u,\theta)^t= 0 $ is smooth by elliptic regularity; the radial expression of the PDE's is a system of ODE's, with singular terms $\frac{u'}r, \frac{\theta ' }r$. Those are regularised by the condition $\theta' (0) = u'(0) = 0 $, and by an adaptation of Cauchy Theorem, there exists a unique solution for any fixed initial datum $\theta(0), u(0)$ (see \cite{ODE}). \\Let $(u_j, \theta_j)$ be a bounded sequence in $H^1_{rad}$ with $\mathcal{L}^{(v_j,\phi_j)}_1 (u_j,\theta_j)^t= 0 $. Boundness in $H^1_{rad}$ implies strong convergence in $L^p$; elliptic regularity and strong convergence of $(v_j, \phi_j)$ leads to strong convergence in $H^1$ to a radial solution of  $\mathcal{L}^{(v,\phi)}_1 (u,\theta)^t= 0 $. Finally, the dimension of $\mathrm{Ker}\mathcal{L}_1^{(v, \phi)} \cap \left ( H^1_{rad} \times H^1_{rad} \right )$ determines the dimension of the variety in a neighbourhood of $(v,\phi)$ and hence it must coincide, for $j$ large, with the dimension of $\mathrm{Ker} \mathcal{L}_1^{(v_j, \phi_j)} \cap \left ( H^1_{rad} \times H^1_{rad} \right )$.
\end{proof}
In the following Lemmas we prove the bound for $\mathcal{L}_1$. By an adaptation of the Rayleigh quotient, the proof narrows down to nonexistence of a preimage through $\mathcal{L}_1$ of the ground state which annihilates the quadratic form. In \cite{W}, \cite{DeB} the explicit solution allows to check the condition directly, while \cite{W2} shows nonexistence from the condition on the $\sigma $ derivative of $\| \varphi_\sigma\|_{L^2}$. Following the same strategy, we need to adapt to the low regularity of the monotonous relation between frequency and charge. For this we have the following variational strategy
\begin{lemma}
    Let $a_0$ as in Theorem \ref{min}, and $a>a_0$. If the map $\sigma (a) $ verifies $\bar{D}^+ \sigma (a)< \infty$, then there exist no $(v,\phi) \in \mathcal{M}^{\sigma(a)}$ $\eta, \theta \in H^1$ such that
    \begin{equation}
    \mathcal{L}_1^{(v,\phi)}   \begin{pmatrix}
 \eta \\ \theta \end{pmatrix}=\begin{pmatrix}
 v \\ 0 \end{pmatrix}; \ \ \ \ 
 \label{app4}
\left \langle\mathcal{L}_1^{(v,\phi)}(\eta,\theta)^t , (\eta,\theta)^t\right \rangle  =0; \ \ \ \ (\eta, v)_{L^2}= 0     \end{equation}
\end{lemma}
\begin{proof}
    Assume by contradiction there exists, for a $(v, \phi) \in \mathcal{M}^{\sigma(a)}$, $\eta, \theta$ as in the Theorem. By energy comparison we infer that the third derivative of $S$ evaluated at $v,\phi $ satisfies
    \begin{equation}
    \label{S'''}
        \left \langle S'''|_{v,\phi} (\eta, \theta)^t, (\eta,\theta)^t, (\eta,\theta)^t\right  \rangle \neq 0
    \end{equation}
   If this were not true for $\varepsilon $ small define
\begin{equation*}
    u_\varepsilon:= v + \eta_\varepsilon; \ \ \ \theta_\varepsilon:= \phi + \varepsilon \theta; \ \ \ \eta_\varepsilon:= \varepsilon  \eta - \tilde{\varepsilon}\frac{\|\eta\|_{L^2}}{\|v\|_{L^2}} v
\end{equation*}
By $(v, \eta)_{L^2}=0$ there exists $\tilde{\varepsilon} = C \varepsilon^2 + o(\varepsilon^2) $ with $C>0$ such that $\|v\|_{L^2}^2= \|u_\varepsilon\|_{L^2}^2$; henceforth
\begin{equation}
    S(u_\varepsilon,\theta_\varepsilon) \ge S(v,\phi)
    \label{contr}
\end{equation}
On the other hand, mimicking \eqref{30} we get
\begin{align*}
   &S(u_\varepsilon, \theta_\varepsilon)- S(v,\phi)   = \frac 1 2 \left \langle \mathcal{L}_1^{(v,\phi)} ( \eta_\varepsilon, \varepsilon\theta) ^t  , ( \eta_\varepsilon, \varepsilon\theta) ^t  \right \rangle+  \frac{\varepsilon^3}6  \left \langle S'''|_{v,\phi} (\eta, \theta)^t, (\eta,\theta)^t, (\eta,\theta)^t\right  \rangle  + O(\varepsilon^4)\\
 & = \frac {\varepsilon^2} 2  \left \langle \mathcal{L}_1^{(v,\phi)} ( \eta, \theta) ^t  , ( \eta, \theta) ^t  \right \rangle - \varepsilon \tilde{\varepsilon}\left \langle \mathcal{L}_1^{(v,\phi)} ( \eta, \theta) ^t  , ( v, 0) ^t  \right \rangle+\frac{\tilde{\varepsilon}^2}{2} \left \langle \mathcal{L}_1^{(v,\phi)} ( v,0) ^t  , ( v,0) ^t  \right \rangle  + O(\varepsilon^4)
\\&= - C\varepsilon^3 \left \langle \mathcal{L}_1^{(v,\phi)} ( \eta, \theta) ^t  , ( v, 0) ^t  \right \rangle + O(\varepsilon^4) = - C \varepsilon^3  \|v\|_{L^2}^2 + O(\varepsilon^4) 
\end{align*}
In the third equality we have used \eqref{app4} and the explicit $\mathcal{L}_1^{(v,\phi)} (v, 0)^t = (0, -4\cos(2\phi)v^2) ^t$, and again \eqref{app4} in the last one. But this contradicts \eqref{contr} for $\varepsilon$ small.\\
Up to a change of sign of $\varepsilon$ in the following argument, we can assume \eqref{S'''} to be strictly negative. Consider now the configuration 
\begin{equation*}
    v_\varepsilon:= v + \frac{\varepsilon} {\|\eta\|_{L^2}}\eta; \ \ \ \phi_\varepsilon:= \phi + \frac{\varepsilon} {\|\eta\|_{L^2}}\theta; 
\end{equation*}
Notice $\|v_\varepsilon\|_{L^2}^2 = a +\varepsilon^2$; modification of the Taylor expansion of the action in the previous calculation for the energy leads to
\begin{equation*}
    J_{a+ \varepsilon^2} \le E(v_\varepsilon, \phi_\varepsilon) = J_a - \sigma(a)\frac{ \varepsilon^2}2 - c \varepsilon^3 + o(\varepsilon^3) 
\end{equation*}
for a positive $c$. On the other hand, by hypothesis on the derivative of $\sigma(a) $ we have $\sigma(a+ \varepsilon^2) \le \sigma (a) + \varepsilon^2 C + o(\varepsilon^2)$, for a finite $C$. But then equation \eqref{endif} reads, for $a_1= a+ \varepsilon^2$, $a_2=a$
\begin{equation*}
    J_{a+\varepsilon^2} \ge J_a - \frac{\varepsilon^2}{2} \sigma(a+\varepsilon^2)\ge J_a - \frac{\varepsilon^2}{2} \sigma(a) + O(\varepsilon^4)
\end{equation*}
and this reaches the contradiction. 
\end{proof}
\begin{lemma} 
If 
    \begin{equation*}
\inf_{\substack{(v,\phi) \in \mathcal{M}^\sigma\\ \eta , \theta \mathrm{\ verifying\ } \eqref{cond4}- \eqref{cond3}}}         \left \langle\mathcal{L}_1^{(v,\phi)}(\eta,\theta)^t , (\eta,\theta)^t\right \rangle \le 0 
    \end{equation*}
\label{ult}
then there exist $(v,\phi) \in \mathcal{M}^\sigma$, $\eta, \theta \in H^1 $ which satisfy \eqref{app4}. 
\end{lemma}
\begin{proof}
    By Remark \ref{para2} and $(\eta, v)_{L^2}= 0$, $\mathcal{L}_1 \ge 0$. Assume the infimum is $0$ and let $(v_j, \phi_j) \in \mathcal M ^\sigma$, $\eta_j, \theta_j\in H^1$ be such that
    \begin{align}
        0 &= \lim_{j\to \infty } \left \langle\mathcal{L}_1^{(v_j, \phi_j)
        }(\eta_j,\theta_j)^t , (\eta_j,\theta_j)^t\right \rangle; \ \  \|  \eta_j\|_{H^1}^2 + \|\theta_j\|_{H^1}^2 =1; 
        \notag\\
        \  \ &  \left((\eta_j, \theta_j)^t,(\partial_l v_j, \partial_l \phi_j)^t \right )_{H^1\times H^1} = 0 = ( \eta_j, v_j)_{L^2}= \left((\eta_j, \theta_j)^t,(k_m^j, h_m^j)^t \right )_{H^1\times H^1} 
        \label{app}
    \end{align}
with $(k_m^j, h_m^j)$ o.n. basis of Ker$\mathcal{L}_1^{(v_j,\phi_j)} \cap \left ( H^1_{rad} \times H^1_{rad} \right )$.
\\$(v_j,\phi_j) \to (v,\phi) \in \mathcal{M}^\sigma$ strongly in $H^2$; by Corollary \ref{corM}, up to subsequences, $ (k_m^j, h_m^j) \to (k_m, h_m)$ o.n. basis of Ker$\mathcal{L}_1^{(v,\phi)} \cap \left ( H^1_{rad} \times H^1_{rad} \right )$ strongly and $\eta_j, \theta_j \to \eta, \theta$ weakly in $H^1$. Orthogonality conditions in \eqref{app} are preserved in the limit, and 
\begin{equation}
\label{y}
    \lim_{j\to \infty} \left ( 2\int \eta_j^2 \sin (2\phi_j) + 4\int v_j\cos(2\phi_j) \eta_j \theta_j\right )= 2 \int \eta^2 \sin (2\phi) + 4\int v\cos(2\phi) \eta \theta
\end{equation}
Hence it follows
\begin{align*}
    0&= \lim_{j\to \infty } \left \langle\mathcal{L}_1^{(v_j,\phi_j)}(\eta_j,\theta_j)^t , (\eta_j,\theta_j)^t\right \rangle \ge 
    \\ &\ge \min \left \{ 2\sigma, 1\right \} + \liminf_{j\to \infty}L_2(\theta_j,\theta_j)- \limsup_{j\to \infty} \left ( 2\int \eta_j^2 \sin (2\phi_j) + 4\int v_j\cos(2\phi_j) \eta_j \theta_j\right ) \ge 
    \\ &\ge \min \left \{ 2\sigma, 1\right \} - 2 \int \eta^2 \sin (2\phi) - 4\int v\cos(2\phi) \eta \theta
\end{align*}
In particular we have $\eta \neq 0  $. Combining Remark \ref{para2}, weak convergence and equation \eqref{y}, $\eta, \theta$ satisfy 
\begin{equation*}
    0 \le     \left \langle\mathcal{L}_1^{(v,\phi)}(\eta,\theta)^t , (\eta,\theta)^t\right \rangle \le  \lim_{j\to \infty } \left \langle\mathcal{L}_1^{(v_j,\phi_j)}(\eta_j,\theta_j)^t , (\eta_j,\theta_j)^t\right \rangle =0  
\end{equation*}
Up to a proper rescaling we have 
 \begin{align}
        0 &=  \left \langle\mathcal{L}_1^{(v, \phi)
        }(\eta,\theta)^t , (\eta,\theta)^t\right \rangle; \ \  \|  \eta\|_{H^1}^2 + \|\theta\|_{H^1}^2 =1; 
        \notag\\
        \  \ &  \left((\eta, \theta)^t,(\partial_l v, \partial_l \phi)^t \right )_{H^1\times H^1} = 0 = ( \eta_j, v)_{L^2}= \left((\eta, \theta)^t,(k_m, h_m)^t \right )_{H^1\times H^1} 
        \label{app2}
    \end{align}
There exists hence Lagrange multipliers $\alpha, \beta_l, \gamma_m, \delta \in \R$ such that
\begin{equation}
    \mathcal{L}_1   \begin{pmatrix}
 \eta \\ \theta \end{pmatrix}= \alpha \begin{pmatrix}
 \eta -\Delta \eta\\ \theta-\Delta \theta \end{pmatrix} + \beta_l \begin{pmatrix}
 \partial_l v -\Delta \partial_l v  \\ \partial_l \phi -\Delta \partial_l \phi\end{pmatrix} + \gamma_m \begin{pmatrix}
 k_m -\Delta k_m  \\ h_m -\Delta h_m\end{pmatrix}+ \delta \begin{pmatrix}
 v \\ 0 \end{pmatrix}
 \label{can}
\end{equation}
Following \cite{W}, Proposition 2.9, or \cite{DeB}, Lemma 2.1, we multiply equation \eqref{can} by $(\eta,\theta)^t$; integrating by parts, conditions in \eqref{app2} imply $\alpha=0$. By Lemma \ref{LemKer}, the same argument gets, with test function $\beta_l(\partial_l v, \partial_l \phi)^t$, $\beta_1=\beta_2 =0$. Similarly $\gamma_m$ is $0$ for any $m$; the only additional difficulty is noticing that $(k_m, h_m)$ is orthogonal to $(v,0)$ since $\mathcal{M}^\sigma \subset S_a$ by Corollary \ref{corM}. 
\\Finally, $\delta\neq 0 $, as otherwise $(\eta, \theta ) \in\mathrm{Ker \ }\mathcal{L}_1^{(v,\phi)} $.
\end{proof}
At this point Proposition \ref{S''} follows by Lemmas \ref{prim}-\ref{ult}.
\subsection{Variational Stability}
We recall, without proof, the classical Concentration-Compactness Lemma (see \cite{Lio}) and prove a Lemma characterizing the angle energy in the splitting scenario. With these tools we can prove Proposition \ref{teost2}.
\begin{lemma}[Concentration-Compactness]
\label{concomp}
    Consider a sequence $u_j\in H^1 \left ( \R^2 \right ) $ such that $\|u_j\|_{L^2}^2 = a>0$. Then there exists a subsequence not relabelled, which verifies one of the following conditions: 
    \begin{itemize}
        \item $\mathrm{Compactness}$ For a certain sequence $x_j \in \R^2$, and $u \in H^1 \left ( \R^2 \right ) $, $u_j( \cdot -x_j)  $ converges strongly to $u$ in $L^p $ for any $p\in [2, \infty)$
        \item $\mathrm {Vanishing}$ $u_j$ converges strongly to $0$ in $L^p $ for any $p\in (2,\infty)$
        \item $\mathrm {Splitting}$ There exists $0<b<a$, such that for any $\varepsilon>0$ there exist two sequences $ u_j^1. u_j^2 \in H^1 \left ( \R^2 \right )$ with compact supports and satisfying, for $j$ sufficiently large
        \begin{align*}
   & \left \|u_j^1 \right \|_{H^1} +    \left \|u_j^2 \right \|_{H^1} \le 4 \sup_j  \left \|u_j \right \|_{H^1}; \  \  \left \|u_j^1+ u_j^2 - u_j \right \|_{L^2} \le \varepsilon ; \\ & \left \| \nabla u_j^1 \right \|_{L^2} +    \left \|\nabla u_j^1 \right \|_{L^2} \le  \left \|\nabla u_j \right \|_{L^2 } + \varepsilon 
    \\ &\left |\int   \left | u_j^1 \right | ^2 \,dx - b \right | \le \varepsilon; \ \ \ \  \left |\int \left | u_j^2 \right | ^2 \,dx +b-a \right | \le \varepsilon; \ \ \ \ \ \mathrm{dist (supp} \ u^1_j, \mathrm{supp\ }u^2_j) \ge \frac2 \varepsilon \notag
        \end{align*}
\end{itemize}
\end{lemma}
\begin{lemma}
For any $u_1, u_2\in H^1$ having compact support, with $\mathrm{dist (supp} \ u^1_j, \mathrm{supp\ }u^2_j)=  \alpha$ we have
\begin{equation*}
    E(u_1 +u_2, \Theta(u_1+u_2)) \ge E(u_1, \Theta(u_1)) + E(u_2, \Theta(u_2)) + C_\alpha
\end{equation*}
with $C_\alpha \to 0 $ for $\alpha \to \infty$, uniformly for $u_1, u_2$ bounded in $H^1$. 
\end{lemma}
\begin{proof}
By contradiction, assume there exist $\alpha_n \to  \infty$, $\tau<0$, $u_1^n, u_2^n$ bounded in $H^1$ and verifying the hypothesis such that for all $n$
    \begin{equation*}
    E(u_1^n +u_2^n,  \theta_n )  -E(u_1^n,\varphi_1^n) - E(u_2^n,  \varphi_2^n)  < \tau 
\end{equation*}
for any sequences $\varphi_1^n, \varphi_2^n$, where $\theta_n:= \Theta\left ( u_1^n +u_2^n\right)$.\\
By \eqref{0an} $\| \theta_n\|_{H^1}^2 \le C$ for any $n.$ There exists a sequence of rays $R_n<\alpha_n/2$ such that over $B_n := B_{R_n+1}\left ( \mathrm{ \ supp \ } u_1^n\right ) \setminus B_{R_n}\left ( \mathrm{ \ supp \ } u_1^n\right )$
\begin{equation*}
    \| \theta_n  \|_{H^1(B_n)} \to 0 
\end{equation*}
where $B_R(K)$ is the ball of radius $R$ of the set $K$. For $\chi_n $ a smooth cut off with $\chi_n\equiv 1$ over $B_{R_n }\left ( \mathrm{ \ supp \ } u_1^n\right ) $,   $\chi_n\equiv 0$ over $B_{R_n+1}^C\left ( \mathrm{ \ supp \ } u_1^n\right ) $ we reach the contradiction by 
\begin{align*}
    &E \left (u_1^n + u_2^n,\theta_n \right ) - E \left (u_1^n , \chi_n \theta_n \right ) -  E \left ( u_2^n, (1-\chi_n) \theta_n \right ) =\\ &=  E^+(\theta_n) - E^+( \chi_n \theta_n) - E^+((1-\chi_n) \theta_n)  \ge- C \|\theta\|_{H^1(B_n)}
\end{align*}
\end{proof}
\begin{proof}[Proof of Proposition \ref{teost2}]
We prove that any minimizing sequence converges to a ground state. Orbital stability in the sense of Definition \ref{defst2} then follows by a classical argument of \cite{sss}.\\
Let $v_n, \phi_n$ be a minimizing sequence over $S_a$. 
The vanishing scenario of Lemma \ref{concomp} cannot happen: if $v_n \to 0$ strongly in $L^4$ we have $\int |v_n|^2 \sin(2\phi_n) \to 0$, but this implies $\liminf E(v_n, \phi_n) \ge 0$, in contradiction with the minimality assumption.
\\We turn now to the splitting scenario for two sequences $v_n^1, v_n^2$. Applying the previous Lemma we get 
\begin{align}
\label{sg}
   \notag E(v_n, \phi_n) &\ge E \left (v_n^1 + v_n^2, \phi_n \right ) + C\varepsilon \ge E \left (v_n^1 + v_n^2, \Theta ( v_n^1 + v_n^2 ) \right ) + C\varepsilon\\ &\ge E \left (v_n^1, \Theta ( v_n^1 ) \right ) +E \left (v_n^2, \Theta ( v_n^2 ) \right ) + C\varepsilon + C_{\frac{1} \varepsilon} \ge J_{b-a}+ J_b + C\varepsilon+ C_{\frac{1} \varepsilon} 
\end{align}
The last inequality follows by a normalization operation, as in \cite{aaa}. The minimal energy $J_a$ is strictly sub additive with respect to $a$: to this end it is enough to iterate the inequality, for $(v_a, \phi_a)$ a ground state over $S_a$ and $b>a$: 
\begin{equation*}
 J_b \le E\left ( \sqrt{\frac{b}{a}} v_a , \Theta    \left (\sqrt{\frac{b}{a}} v_a\right ) \right ) < E\left ( \sqrt{\frac{b}{a}} v_a , \phi_a \right ) \le \frac{b}a J_a 
\end{equation*}
But for $\varepsilon \to 0$ \eqref{sg} contradicts strict subadditivity.
\\Hence $v_n$ must converge, up to translation, to $v$ strongly in $L^p$ for $p>2$. Weak convergence in $H^1$ of $\phi_n$ is sufficient for convergence of the integral $\int |v_n|^2\sin (2\phi_n)$. Strong convergence to a ground state then follows the same step as in the proof of Theorem \ref{min} in \cite{Arg}. 
\end{proof}
\section{Existence for any $0<\sigma<1$}\label{secexi}
In this section we prove Theorem \ref{teosig}: we lay down the variational problem, define the Nehari manifold and prove the existence of the minimizer.\\ In the definition of a modified energy $E_\sigma$, we are naturally directed by the request that a stationary point for the energy is a stationary solution of the problem \eqref{sigma}-\eqref{sig}. For the construction of the Nehari manifold, the idea is to look for a solution of $E_\sigma'(u)=0$ between the functions that satisfy $E'_\sigma(u)(u)=0 $. This will be the equation describing the constraint.\bigskip\\
In the presence of favourable growth and smallness condition for the functional $E_\sigma$, as well as some regularity, the Nehari manifold is a convenient instrument to prove the existence of critical points. The interested reader is directed to the user friendly notes \cite{Nene} for a wider presentation of the argument.\\Regarding our problem, the regularity of $E_\sigma$ can be proved with a bit of effort. The growth condition is more critical, as the functional does not match the standard ones present in the literature. Still, we can recover some weak form of compactness by careful energetic comparisons around the minimal value of the constrained problem.\bigskip\\
We fix hereafter a value $0<\sigma <1$. We define the modified Energy as 
\begin{equation}
    \label{enemodi}E_\sigma: H^1(\R^2; \R) \to \R; \ \ \ \ E_\sigma(u) := E^-(u,\Theta(u)) + \frac{\sigma}{2} \|u\|_{L^2}^2  + E^+( \Theta(u)) 
\end{equation}
For the minimization problem, we are restricting to real valued and radially symmetric functions. In particular the second request seems quite important, as we need  in several steps the compact embedding $H^1_{rad} \hookrightarrow L^p$ for $2<p<\infty$. Recovering radial symmetry starting with functions in $H^1$ through rearrangements represents a non trivial complication, hence we prefer to consider the variational problem in the simpler setting $H^1_{rad}$.\bigskip\\ Moreover, in the definition of the energy we are absorbing the dependence on the $\theta$ variable in  $\Theta(u)$, the minimizing angle by Lemma \ref{minan}. We recall that for $u\in H^1_{rad}$ $\Theta(u)$ is radial. \\This choice actually contributes to get the desired compactness, which is the most delicate part, but comes at the cost of some additional technicalities linked to the implicit map $\Theta$ and its low regularity as a map from $H^1 \to H^1$.\\In particular, we need to spend some calculations to prove the explicit formula of the derivative of the energy. If $\Theta$ were $C^1$, this would have come directly from
\begin{equation*}
    \frac{d}{du} E_\sigma (u) = \frac{\partial}{\partial u }E_\sigma(u)+ \frac{\partial}{\partial \Theta } E(u,\Theta(u))\frac{\partial \Theta }{\partial u} = \frac{\partial}{\partial u } E_\sigma (u)
\end{equation*}
where the last equality is just \eqref{eqcal2}. As we do not know if $\Theta$ is $C^1$, we have the following
\begin{proposition}
 The energy $E_\sigma$ is $C^1\left (H^1_{rad};\R\right )$ and its derivative is given by
 \begin{equation}
     \label{enederi} E_\sigma '(u)v= \frac{1}{2}\left ( \int \nabla u\cdot \nabla v + 2\sigma \int uv - 2 \int uv \sin(2\Theta(u)) \right ) 
 \end{equation}
\end{proposition}
\begin{proof}
    We decompose the difference of the energy in $u+v$ and $u$ as
    \begin{align}
        \notag 4&(E_\sigma(u+v) - E_\sigma (u)) = \int |\nabla (u+v)|^2 - |\nabla u |^2 + \left (|u+v|^2 -|u|^2\right )\left(2\sigma - 2 \sin(2\Theta(u)\right) +\\ &+ 2\int |u+v|^2\left(\sin(2\Theta(u)) - \sin (2\Theta(u+v)\right) + 4 (E^+ (\Theta(u+v)) - 4E^+ (\Theta(u))) \label{aranc}
    \end{align}
    The first integral is known to be differentiable, thus dividing by $\|v\|$ and passing to the limit leads to the desired result by standard calculations. It only remains to consider the second line, which we relabel as $I$. $\Theta (u), \Theta (u+v)$ solve equation \eqref{eqcal2}; multiplying the sum of the two equations by $\Theta (u+v)- \Theta (u)$, we can rewrite the angular Dirichlet energy as
    \begin{align*}
        \lambda \int |\nabla &\Theta(u+v)|^2 - | \nabla \Theta(u)|^2 = -  q\int (\sin(2\Theta (u+v)) + \sin (2\Theta(u))) (\Theta(u+v)- \Theta(u))+\\ +&2\int \left ( |u|^2 \cos(2\Theta(u)) + |u+v|^2 \cos (2\Theta(u+v))\right )( \Theta(u+v)- \Theta(u) ) 
    \end{align*}
Using this equality Taylor expansions for the trigonometric functions infers for the remaining term in \eqref{aranc}
\begin{equation*}
 I = O\left (\|v\|_{L^2}\| \Theta(u+v)- \Theta(u) \|_{L^2} + \| \Theta(u+v)- \Theta(u) \|_{L^2}^2 \right)
\end{equation*}
By estimate \eqref{scad} the previous term reads
\begin{equation}
\label{mando}
     I \le  C_u O \left(\|v\|_{L^2}\| v \|_{L^4} + \| v \|_{L^4}^2 \right)
\end{equation}
which goes to zero faster than $\|v\|_{H^1}$.
\end{proof}
\begin{remark}
\label{scrib}
If $u$ satisfies $E'_\sigma(u)=0$, then the couple $\left(u,\Theta(u)\right )$ is a solution of the system \eqref{sigma}-\eqref{sig} with the fixed value of $\sigma$
\end{remark}
Now we turn to the definition of the Nehari manifold, and its main properties. 
\begin{proposition}
    Let $\mathcal{N}_\sigma$ be the Nehari manifold
    \begin{equation}
 \mathcal{N}_\sigma:= \left \{ 0 \neq u \in H^1_{rad} \ \biggm| \ \int |\nabla u |^2 + 2\sigma |u|^2 - 2 \sin(2\Theta (u))|u|^2\,dx =0 \right \}   \end{equation}
 Then the following hold:
 \begin{enumerate}
\item $ \mathcal{N}_\sigma$ is complete in $H^1_{rad}$
     \item There exists a continuous one to one correspondence $m: S^\sigma \to \mathcal{N}_\sigma$, for $S^\sigma $ defined as 
     \begin{equation}
         S^\sigma := \left \{ u \in H^1_{rad}; \ \|u\|_{H^1}=1 \ \biggm| \ \int |\nabla u |^2 - 2(1-\sigma) |u|^2 \,dx <0 \right \}
     \end{equation}
     The map is a local homeomorphism and its inverse is given by 
     \begin{equation}
         m^{-1}(w)= \frac{w}{\|w\|_{H^1}}
     \end{equation}
     \item There exists a $\delta>0$ such that $\|u\|_{H^1} \ge \delta$ for any $u \in \mathcal{N_\sigma}$
     \item For any compact subset $W \subset S^\sigma$, there exists a constant $c_W$ such that the image of $m(W)$ is contained in $B_{c_W}(0)$
     \end{enumerate}
     \label{Neah}
\end{proposition}
We omit hereafter the indices $\sigma$ for $S, \mathcal N$ without any risk of confusion, while we keep the subscript for the energy $E_\sigma$, in order not to confuse it with the energy $E$. 
\begin{proof}
Completeness of $\mathcal{N}$ follows from completeness of $H^1_{rad}$, continuity of $\Theta$ given by Lemma \ref{minan} and $0\notin \overline {\mathcal{N}}$. The last property is a consequence of point $3$ in the Proposition, whose proof does not rely on completeness.\\
Any $u \in \mathcal{N}$ satisfies 
\begin{equation*}
    \int |\nabla u |^2 - 2(1-\sigma) |u|^2 \,dx \le 0 
\end{equation*}
If it were an equality for some $u\in \mathcal N $, then it would imply $\Theta(u) \equiv \pi/4$ over all the support of $u$; but this cannot be as already commented in the proof of Lemma \ref{leman}. Hence, $u/\|u\| \in S$.\bigskip\\On the other hand, for $w \in S$, consider the function $\alpha_w(r) := E_\sigma( rw)$ for $r>0$, with derivative $\alpha'_w(r) = E_\sigma'(rw)w$. We claim that there exist a unique $r_w$ such that $ \alpha'_w(r_w)=0$ and hence $r_ww \in \mathcal N$; in particular $\alpha'_w(r)>0$ for $0<r<r_w$ and $\alpha'_w(r)<0$ for $r>r_w$. \\For $r \to 0$, we have $\alpha_w(r)\to 0$ and, since $\sigma $ is positive, $\alpha_w(r)$ must be positive for all $r< \tilde{a}$, $\tilde{a}$ as in Theorem \ref{min}. Moreover for $w \in S$ fixed, we can find an angle $\phi$ such that 
\begin{equation*}
    \int |\nabla w |^2 + 2\sigma |w|^2 - 2 \sin(2\phi)|w|^2\,dx =-c<0
\end{equation*}
For $r \to \infty$ therefore, by Lemma \ref{minan}, we have
\begin{equation}
\label{confo}
    \alpha_r(w) \le E(rw, \phi) + 2\sigma \|rw\|_{L^2}^2  \le -cr^2 + E^+ (\phi)
\end{equation}
and the right hand side goes to $-\infty$ for $r$ large. Hence there exists at least one $r_w$ such that $\alpha'_w(r_w)=0$.\\ By the minimality of $\Theta$ in
Lemma \ref{minan} we have that for any $r>0$, for any $\phi \in H^1$
\begin{equation*}
    \int |\nabla w|^2 +2(\sigma- \sin (2\Theta(wr)))|w|^2   + \frac{E^+(\Theta(wr))}{r^2} \le \int |\nabla w|^2 +2(\sigma  -  \sin (2\phi)|w|^2   + \frac{E^+(\phi)}{r^2}
\end{equation*}
Applying the relation twice, since $r_ww\in \mathcal{N}$
\begin{align*}
& \alpha_w(r)=  r^2\left( \int |\nabla w|^2 +2\sigma |w|^2 - 2 \sin (2\Theta(wr))|w|^2\right)   + E^+(\Theta(wr)) \le \\ & \le E^+(\Theta(wr_w)) \le r_w^2 \left( \int |\nabla w |^2 + 2\sigma |w|^2 - 2 \sin(2\Theta (wr))|w|^2\,dx\right ) + E^+(\Theta(wr))
\end{align*}
Rearranging terms we get that if $\alpha_w'(r)=\frac{r}{2} \int|\nabla w|^2 +2\sigma |w|^2 - 2 \sin (2\Theta(wr))|w|^2  $ is positive, then $r<r_w$, and if it is negative $r>r_w$. \\It only remains to exclude that $\alpha'_w(r)=0$ for $r$ in an interval $(r_1, r_2)$. But this follows easily from the uniqueness of $\Theta(u)$ as a minimizer: if $\alpha_w'(r_1)= \alpha'_w(r_2)= 0$, then $\Theta(r_1w) = \Theta(r_2w)$. Since the angles solve equation \eqref{eqcal2}, by taking the difference it would imply $0= (r_1^2 - r_2^2)\cos (2\Theta(r_1w))| w|^2$, impossible. \bigskip \\To prove the last points, we notice uniformly in $w \in S$, and for $r$ small we have 
\begin{equation*}
    \alpha_w'(r)= \frac{r}{2} \int|\nabla w|^2 +2\sigma |w|^2 - 2 \sin (2\Theta(wr))|w|^2 \ge \frac{r}{2} \int|\nabla w|^2 +2\sigma |w|^2 - 2O(r^\frac 3 2 )\|w\|_{L^4}^3
\end{equation*}
by estimate $\|\Theta(rw)\|_{H^1} \le C \|rw\|_{L^4} $ of Lemma \ref{minan} and Holder inequality. Hence exists a uniform bound $0<\delta \le r_w $ for $w \in S$. Finally, for any compact subset $W \subset S$ we have uniformly $ \int|\nabla w|^2 -2(1-\sigma) |w|^2 \le c_W <0  $, so that we can repeat the comparison done in \eqref{confo} to get $r_w \le C_W $ for any $w\in W$. 
\end{proof} 
The previous result allows us to apply an abstract result for Nehari manifolds, and to study the functional as if it were defined over $S$. In the following we denote, for $M$ a $C^1$ manifold in $H^1_{rad}$, $T_{M,p}$ as the tangent space to $M$ in $p \in M$.
\begin{lemma}
\label{sf}
    Consider the functional $\Psi : S \to \R$ defined as 
    \begin{equation*}
        \Psi ( w) = E_\sigma ( m(w)); \ \ w \in S
    \end{equation*}
where $m$ is the homeomorphism given in Proposition \ref{Neah}. Then the following hold: \begin{enumerate}
    \item $\Psi$ is $C^1$, and it holds
\begin{equation}
    \Psi'(w)(z) = \|m(w)\|  E_\sigma'(m(w))(z) \ \ \forall \ z \in T_{S,w}
    \label{deride}
\end{equation}
\item $w$ is a critical point for $\Psi $ if and only if $m(w) $ is a critical point for $E_\sigma$ over $\mathcal{N}$
\item If $u$ is a minimum point for $E_\sigma$ over $\mathcal{N}$, it satisfies $E'_\sigma(u)=0$
\end{enumerate} 
\end{lemma}
For the proof we refer the reader to \cite{Nene}, Proposition 9 and Corollary 10.
\begin{remark}
\label{notte}
    In the variational problem, we notice that minimizing $E_\sigma$ over $\mathcal{N}$ is equivalent to minimize $E^+(\Theta(u)) $ over the same set, since the condition in the constraint is equivalent to a part of the energy.
\end{remark}
\begin{theorem}
\label{sigma fin}
Consider $0<\sigma<1$, and let $c$ be the infimum of the energy over the Nehari manifold:
\begin{equation*}
    c= \inf_{u \in \mathcal N_\sigma} E_\sigma(u) 
\end{equation*}
Then there exists a radial decreasing $v \in \mathcal{N}_\sigma $ such that $E_\sigma(v) = c$. The minimum $v$ and the angle $\phi= \Theta(v)$ are solutions of the ground state equations \eqref{sigma}-\eqref{sig} for the given value of $\sigma$. 
\end{theorem}
Before proving the Theorem, we claim the existence of a minimizing sequence with vanishing derivative
\begin{lemma}
    There exists a minimizing sequence $u_n \in \mathcal N$ for $E_\sigma $ such that $E_\sigma'(u_n)\to 0$. 
\end{lemma}
\begin{proof}
    $\mathcal{N}$ is complete and $E_\sigma$ is positive over $\mathcal{N}$; by Ekeland's variational principle (see Theorem 4.1 of \cite{Eke}), for any $\varepsilon>0$ there exist $u_\varepsilon\in \mathcal N$ such that
\begin{align}
    \notag &E_\sigma(u_\varepsilon) \le c+ \varepsilon;\\
    \\&E_\sigma(u_\varepsilon) \le E_\sigma(u) + \varepsilon \|u-u_\varepsilon\|_{H^1} \ \ \ \forall \ u \in \mathcal{N}
    \label{Pipi}
\end{align}
By Proposition \ref{Neah}, $u_\varepsilon= m (w_\varepsilon)$ with $w_\varepsilon\in S$. For any $z\in T_{S,w_\varepsilon}$, consider a $C^1$ curve $\gamma(t) \in S$ such that $\gamma(0) = w_\varepsilon$, $\dot{\gamma}(0)=z$. Inserting $u=m(\gamma(t))$ in inequality \ref{Pipi} we have
\begin{equation*}
\Psi(w_\varepsilon) - \Psi(\gamma(t)) \le \varepsilon\| u_\varepsilon- u\|_{H^1} \le C\varepsilon\|u_\varepsilon\|_{H^1} \|w_\varepsilon- \gamma(t)\|
\end{equation*}
The term $\| u_\varepsilon\|$ on the right hand side comes from the scaling of the norm for the local homeomorphism $m$. Now passing to the limit for $t\to 0$ it follows
\begin{equation*}
    \Psi'(w_\varepsilon)z \le C \varepsilon \|u_\varepsilon\|_{H^1}\|z\|_{H^1}
\end{equation*}
The inequality holds for any $z \in T_{S,w_\varepsilon}$, and passing to the supremum we get 
\begin{equation*}
    \frac{\|\Psi '(w_\varepsilon)\|_{T'_{S,w_\varepsilon}}}{\|u_\varepsilon\|} \le C \varepsilon
\end{equation*}Notice that for any $w_\varepsilon\in S$, $H^1_{rad}= T_{S,w_\varepsilon} \bigoplus \R w_\varepsilon$, and that $E_\sigma'(u_\varepsilon) w_\varepsilon =0$ because of the Nehari constraint. At this point the thesis follows by \eqref{deride}.
\end{proof}
\begin{proof}[Theorem \ref{sigma fin}]
    Let $u_n\in \mathcal N$ be the minimizing sequence satisfying $E'_\sigma (u_n) \to 0$. If $u_n$ is unbounded, then up to subsequences $v_n := u_n /\|u_n\|_{H^1}\rightharpoonup v$. If the weak limit were $0$ then for any $r>0$, for $\varepsilon_n \to 0$
    \begin{equation}
    \label{jjj}
      c+\varepsilon_n \ge E_\sigma(u_n) = E_\sigma(r_{v_n}v_n) \ge E_\sigma(r v_n) \ge  \sigma r^2 \|v_n\|_{H^1}^2 - r^2 \int v_n^2 \sin(2\Theta(rv_n))
    \end{equation}
By compact embedding $H^1_{rad}(\R^2) \hookrightarrow L^p(\R^2)$ for any $2<p<\infty$, we have $v_n^2 \to 0 $ strongly in $L^{\frac{3}{2}}$. Similarly, for $r$ fixed, $\Theta(v_nr)$ is bounded in $H^1_{rad}$ and converges strongly in $L^3$ up to subsequences. Hence the integral in the right hand side converges to $0$, and we reach a contradiction by choosing $r$ large enough.\bigskip \\
If $v\neq 0$, by remark \ref{notte} $E^+( \Theta(u_n))\le c+ \varepsilon_n, \ \varepsilon_n \to 0$; up to subsequences, $\Theta (u_n) \rightharpoonup \phi$ in $H^1_{rad}$. Moreover the constraint equality which defines $\mathcal {N}$ passes to the limit as an inequality by weak convergence: 
\begin{equation}
\label{papapa}
    \int |\nabla v_n|^2 + 2\sigma |v_n|^2 - 2\int |v_n|^2 \sin (2\Theta( u_n)) = 0 \ge  \int |\nabla v|^2 + 2\sigma |v|^2 - 2\int |v|^2 \sin (2\phi) 
\end{equation}
As before, the negative integral converges because of the compact embeddings of $H^1_{rad}$. \\We can prove that the above inequality is actually an equality, and hence the convergence $v_n \to v$ is strong in $H^1$. In fact, by the Palais Smale condition, we have 
\begin{equation}
    0= \lim_{n} \frac{4E'(u_n) v}{\|u_n\|} =   \int \nabla v_n \cdot \nabla v + 2\sigma \int v_n v - 2\int v_n v \sin (2\Theta( u_n))
    \label{pepe}
\end{equation}
The integral on the right hand side converges by weak convergence to the right hand side of equation \eqref{papapa}.\\We have the following asymptotic for the energy for large $\|u_n\|_{H^1}$:
\begin{equation}
\label{asimp}
    \frac{E_\sigma(u_n) }{\|u_n\|_{H^1}^2} = \int |\nabla v|^2 + 2\sigma |v|^2 - 2|v|^2 \sin (2\Theta(u_n)) + \frac{E^+ (\Theta(u_n))}{\|u_n\|_{H^1}^2} + o(1 ) 
\end{equation}
The last term represents the difference between $v$ and $v_n$, and goes to $0$ with $n$ by strong convergence. By this explicit formulation of the leading term in the expansion, in the limit we expect $\Theta(u_n)$ to maximize the negative contribution of the energy, $\int|v|^2 \sin (2\Theta_n)$.\\Recall again $\Theta(u_n)$, bounded in $H^1$, minimizes the energy for $u_n$ fixed. We claim that $v$ has compact support, and $\phi\equiv \pi/4$ on the support of $v$. If this were not true, we could take a positive $\alpha \in H^1$ such that $\theta := \alpha + \phi $ verifies $\theta \le \frac{\pi }4$ and the set \begin{equation*}
     \left \{ \theta> \phi \right \} \cap \left \{ v>0\right \}    
\end{equation*} has positive measure.
Then for $C>0$ we have
\begin{equation*}
    2\int |v|^2 \sin (2\phi) +2C   \le  2\int |v|^2 \sin (2\theta) 
\end{equation*}
This implies definitively in $n$, for a certain $o_n(1) \to 0 $ for $n \to \infty$
\begin{equation*}
    2\int |v|^2 \sin (2\Theta(u_n)) +C=2\int |v|^2 \sin (2\phi) +C + o_n(1)   \le  2\int |v|^2 \sin (2\theta) 
\end{equation*}
But this, combined with \eqref{asimp}, contradicts the minimality of $\Theta(u_n)$ for $n$ large. 
\\By weak convergence we have $E(\phi) \le \lim E^+(\Theta(u_n)) = c$; but we can prove explicitly that such a $\phi$ cannot have minimal energy. Consider $v^*, \phi^*$ symmetric decreasing rearrangements of $v,\phi$. By Pólya–Szegö inequality and properties for symmetric rearrangements, we have $E^+(\phi^*) \le c$. Moreover it remains $\phi^*\equiv \pi/4$ over the support of $v^*$, and $v^*$ still verifies
\begin{equation*}
\int |\nabla v^*|^2 - 2(1-\sigma) |v^*|^2 \le 0 
\end{equation*}
Define $R>0$ by $B_R(0)= \mathrm{spt}v^*$; by continuity, there exists a $\varepsilon>0$ such that $\sin (2\phi^*) \ge 1/2 $ on $B_{R(1+\varepsilon)}(0)$. We look at the rescaled function $v_s(x):= v^*(sx)$. We have 
\begin{equation*}
    \int |\nabla v_s|^2 - 2(1-\sigma) |v_s|^2=\int |\nabla v^*|^2 - \frac{2(1-\sigma)}{s^2} |v^*|^2 \le  \frac{2(1-\sigma)(s^2-1)}{s^2} \int|v^*|^2
\end{equation*}
so that $v_s/\|v_s\|$ belongs to $S$ for any $s<1$. For $s<1$ sufficiently close to $1$ it holds
\begin{equation*}
 \int |\nabla v_s|^2 + 2\sigma |v_s|^2 - 2|v_s|^2 \sin (2\phi^*)< 0
\end{equation*}
Fixed such $s$, by Proposition \ref{Neah} there exists a $r=r_{v_s}$ such that $rv_s\in \mathcal{N}$, i.e.
\begin{equation}
\label{r1}
    \int |\nabla v_s|^2 + 2\sigma |v_s|^2 - 2|v_s|^2 \sin (2\Theta(r_{v_s}v_s)) = 0 
\end{equation}
Again by Lemma \ref{minan}, we have 
\begin{equation}
\label{r2}
    E_\sigma(r_{v_s}v_s) \le E(r_{v_s}v_s, \phi^*) + 2\sigma\|r_{v_s} v_s\|_{L^2}^2
\end{equation}
but since 
\begin{equation}
   \int |\nabla v_s|^2 + 2(\sigma - \sin (2\phi^*))|v_s|^2 <
   \int |\nabla v_s|^2 + 2(\sigma - \sin (2\Theta(r_{v_s}v_s))) |v_s|^2 = 0 
   \label{r3}
\end{equation}
this would imply $E^+(\Theta( r_{v_s}v_s))< E^+ (\phi^*) \le c$, absurd by definition of $c$.\bigskip
\\Hence the sequence $u_n$ is bounded, and up to subsequences converges to $u$. Repeating the calculations as in \eqref{jjj}, we deduce $u\neq 0$. Proceeding as in \eqref{papapa}- \eqref{pepe}, we deduce strong convergence to $u\in \mathcal N$. Finally, by Lemma \ref{minan} we have that $u$ is the minimal point.
\\If $u$ is not symmetric, the radial rearrangement $u^*, \phi^*= \left (\Theta(u)\right ) ^*$ would satisfy 
\begin{equation*}
     \int |\nabla u^*|^2 + 2\sigma\int |u^*|^2 - 2\int \sin (2\phi^*))|u^*|^2 < 0; \ \ E^+(\phi^*) \le E^+(\Theta(u))
\end{equation*}
Following the same steps as in \eqref{r1}, \eqref{r2} and \eqref{r3} for a certain $r>0$, $ru^* \in \mathcal N$ and it verifies $E^+(\Theta(ru^*)) < E^+(\Theta(u))= c$, absurd.
\end{proof}
At this point the statement of Theorem \ref{teosig} follows from Theorem \ref{sigma fin}, the last point of Lemma \ref{sf} and Remark \ref{scrib}.
\section{Decaying rate}\label{secdec}
In this section we prove Proposition \ref{decay}. The proof is an adaptation to the coupled system of a standard method, see \cite{BL1}, Lemma 4.2.. The argument is not immediately transparent, as it is involves the use of several auxiliary functions and their associated ODE's; nonetheless once the method is implemented, the proof becomes simple.
\\We recall the following decaying property for radially decreasing functions.
\begin{lemma}
\label{le1}
    Let $f\in L^2\left (\R^2\right)$ be a radial decreasing function. Then there exist $C>0$ such that for any $r$
    \begin{equation*}
        |f(r)|\le \frac{C}{r}\|f\|_{L^2}
    \end{equation*}
\end{lemma}
For the proof see \cite{Caz}, Lemma [1.7.3].\\
\begin{proof}[Proof of Proposition \ref{decay}]
By hypothesis, $(v, \phi)$ are $H^1$ solutions of
\begin{equation}
\label{sno}
        -\Delta v =2v\sin (2\phi) - 2\sigma v; \ \ -\Delta \phi = -q\sin (2\phi) + 2|v|^2 \cos(2\phi)
\end{equation}
The nonlinear terms on the right hand side of each equation are smooth functions depending only on $v, \phi$. By standard elliptic estimates, we can bootstrap higher regularities and infer $v, \phi \in H^k$ for any $k$. \\
We prove the decay estimate for $v$. By Sobolev embeddings $v\in C^2$, and we can write equation \eqref{sno} as an ordinary differential equation in the radial variable
\begin{equation*}
    -v_{rr} -\frac{1}{r} v_r = 2v(r)\sin (2\phi(r)) - 2\sigma v(r)
\end{equation*}
By standard computation (see \cite{BL1}) the function $w(r):= r^{1/4}v^2(r) $ satisfies
\begin{equation}
\label{a1}
    w_{rr}= \left (r^{\frac{1}{2}}v \right )_r^2 + \left [ 2\sigma - 2\sin(2\phi(r)) - \frac{1}{r^2}\right ]w(r) 
\end{equation}
For $r >r_0$ by Lemma \ref{le1} we have 
\begin{equation}
\label{a_2}
    2\sigma - 2\sin(2\phi(r)) - \frac{1}{r^2} \ge \sigma 
\end{equation}
$w\ge 0$ by definition, and $w_{rr}\ge \sigma w$ for $r>r_0$. \\We look then at the last auxiliary function
$z(r)= e^{-\sqrt{\sigma} r}(w_r + \sqrt{\sigma} w)$, that satisfies $z_r\ge 0$ on $(r_0,\infty)$. \\If for some $r_1>r_0$, $z(r_1) >0$, then $w +w_r$ is not integrable in $(r_1, \infty)$ since
\begin{equation*}
    w_r + \sqrt{\sigma} w \ge z(r_1)e^{\sqrt{\sigma}r} 
\end{equation*}
But this cannot happen by definition of $w$: since $v \in H^1(\R^2)$ we have
\begin{equation*}
    \int_{r_1}^\infty rv^2(r) + \int_{r_1}^\infty rv_r^2 < \infty
\end{equation*}
and hence $w_r, w$ are integrable. 
\\ For $r\in (r_0, \infty)$ we must have then $z(r)\le 0$. By definition of $z$ this implies $(e^{\sqrt{\sigma}r} w(r))_r =e^{\sqrt{\sigma}r} z(r)\le 0$, and hence $w(r) \le C e^{-\sqrt{\sigma}r} $. Again by definition of $w$ we have
\begin{equation*}
    |v(r)|\le C r^{-\frac{1}{2}}e^ {-\sqrt{\sigma}r} \ \ \ \forall \ r\ge r_0
\end{equation*}
We turn now to the estimate for $\phi$, which we get by a barrier method. For $r$ large and $\varepsilon<q $ small it holds in $B_r^C(0)$
\begin{equation*}
    -\Delta \phi + 2(q-\varepsilon) \phi \le 2|v|^2 \cos (2\phi)
\end{equation*}
Let $\psi:= e^{-\beta |x|}$ with $\beta < \min\left \{ \sigma, \sqrt{q- \varepsilon}\right \} $. For $C>0$ to be fixed, by straightforward computation and for $r$ large enough
\begin{equation*}
    - \Delta ( \phi - C \psi) + (q-\varepsilon) (\phi - C\psi) \le  2 |v|^2 - C\left ( q- \varepsilon - \beta^2  -\frac{1}{|x|} \right ) \psi  \le 2|v|^2 - Cc_\beta \psi
\end{equation*}
for a positive $c_\beta$. Definition of $\beta $ and exponential decay of $v$ imply there exists $R, C$ large enough such that in $B_R^C(0)$
\begin{equation*}
     - \Delta ( \phi - C \psi) + (q-\varepsilon) (\phi - C\psi) \le 0
\end{equation*}
Moreover we can pick $C$ even larger and get $(\phi - C \psi)(R) \le 0$. By maximum principle we have therefore the desired exponential decay
\begin{equation*}
    \phi(r) \le C e^{-\beta |x|} \ \ \ \ \forall \ |x|  > R
 \end{equation*}
\end{proof}\bigskip \textbf{Acknowledgments}        \bigskip \\
The author has been partially supported by the Basque
Government through the BERC 2022-2025 program and IKUR program, by the project PID2023-146764NB-I00 funded by MICIU/AEI/10.13039/
501100011033 and cofunded by the European Union, and by the Spanish State Research Agency through
BCAM Severo Ochoa CEX2021-001142. \\The author warmly thanks the anonymous referees, as their
patient and careful reading of previous versions has been of significant help for the preparation of the
present work, and A. Zarnescu, for proposing the problem and his help during the elaboration.


\begin{thebibliography}{99}
\bibitem{Ad} R. Adami, R. Carlone, M. Correggi, L. Tentarelli: \textit{Stability of the standing waves of the
concentrated NLSE in dimension two} Math. Eng. 3 (2021).
\bibitem{as2} G. Assanto, N.F. Smyth: \textit {Self-confined light waves in nematic liquid crystals} Physica D (2019).
\bibitem{SP4}  A. H. Ardila: \textit{Orbital stability of standing waves for a system of nonlinear Schrödinger
equations with three wave interaction} Nonlinear Anal., 167, 1–20 (2018).
\bibitem{SP1} J. Bellazzini, G. Siciliano: \textit{Stable standing waves for a class of nonlinear Schrödinger-Poisson equations} Z. Angew. Math. Phys. 62, 267–280 (2011).
\bibitem{BL1} H. Berestycki, P.L Lions: \textit{Nonlinear scalar field equations, I existence of a ground state} Arch. Rational Mech. Anal. (82) 313–345 (1983).
\bibitem{Arg} J. P. Borgna,  Panayotis Panayotaros, D. Rial, C. S. F. de la Vega: \textit{Optical solitons in nematic liquid crystals: model with saturation effects} Nonlinearity 31, 1535-1559 (2018).
\bibitem{Arg2} J. P. Borgna,  Panayotis Panayotaros, D. Rial, C. S. F. de la Vega: \textit{Optical solitons in nematic liquid crystals: Large angle model} Physica D (2020) .
\bibitem{sss} T. Cazenave, P.L. Lions: \textit{Orbital stability of standing waves for some non linear Schrödinger equations} Commun. Math. Phys. 85, 549–561 (1982).
\bibitem{Caz} T. Cazenave: \textit{Semilinear Schrödinger equations} Courant Lecture Notes in Mathematics , vol. 10, New York University, Courant Institute of Mathematical Sciences, New York, American
Mathematical Society, Providence, RI (2003).
\bibitem{Cof} C. V. Coffman: \textit{Uniqueness of the ground state solution for $\Delta u-u+u^3=0$ and a variational characterization of other solutions} Arch. Ration. Mech. Anal. 46, 81–95 (1972).
\bibitem{SP2} M. Colin, T. Watanabe: \textit{Stable standing waves for Nonlinear Schrödinger-Poisson system with a doping profile} preprint- \url{https://arxiv.org/abs/2409.01842}
\bibitem{SP3} M. Colin, T. Watanabe: \textit{A refined stability result for standing waves of the Schrödinger–Maxwell system} Nonlinearity 32 (2019).
\bibitem{ODE} E. A. Coddington, N. Levinson :\textit{Theory of Ordinary Differential Equations} McGraw‑Hill (1955). 
\bibitem{DeB} A. de Bouard, R. Fukuizumi:\textit{ Stability of standing waves for nonlinear Schr\"odinger equations with inhomogeneous nonlinearities} Ann. Henri Poincar\'e{}  6, 1157--1177 (2005).
\bibitem{Eke} G. D. De Figueiredo: \textit{Lectures on the Ekeland variational principle with applications and detours} Vol. 81. Berlin: Springer, (1989).
\bibitem{Rep1} S. De Bièvre, F. Genoud, S. Rota-Nodari: \textit{Orbital stability: analysis meets geometry} Nonlinear optical and atomic
systems, Lecture Notes in Mathematics, vol. 2146, Springer (2015).
\bibitem{Din} V. D. Dinh: \textit{The 3D Nonlinear Schrödinger Equation with a Constant Magnetic Field Revisited} J Dyn Diff Equat 36, 3643–3686 (2024). 
\bibitem{D2} V. D. Dinh, A. Esfahani: \textit{A system of inhomogeneous NLS arising in optical media with a $ \chi^{(2)} $ nonlinearity, part II: Stability of standing waves} Discrete and Continuous Dynamical Systems - B, 30(7): 2209-2232, (2025).
\bibitem{FA0} N. Fukaya:\textit{ Uniqueness and nondegeneracy of ground states for nonlinear Schrödinger equations with attractive inverse-power potential.} Communications on Pure and Applied Analysis, 20(1) (2021)
\bibitem{FA} N. Fukaya, V.S. Georgiev, M. Ikeda:\textit{ On stability and instability of standing waves for 2d-nonlinear Schrodinger equations with point interaction} J. Differential Equations  321, 258-295 (2022).
\bibitem{Fuk} R. Fukuizumi: \textit{ Stability and instability of standing waves for nonlinear Schrödinger equations}
Ph.D thesis, Tohoku University, Sendai, Japan, (2003).
\bibitem{FukO} R. Fukuizumi, M. Ohta:\textit{ Instability of standing waves for nonlinear Schrödinger equations
with potentials} Differential Integral Equations 16, 691–706 (2003).
\bibitem{Fuk1} R. Fukuizumi, M. Ohta: \textit{Stability of standing waves for nonlinear Schrödinger equations
with potentials} Differential Integral Equations 16, 111–128 (2003).
\bibitem{New} F. Genoud, C. A. Stuart: \textit{Schrödinger equations with a spatially decaying nonlinearity: Existence and stability of standing waves.} Discrete and Continuous Dynamical Systems, 21(1) (2008).
\bibitem{GNM} B. Gidas, W. M. Ni, L. Nrenberg: \textit{Symmetry and related properties via the maximum principle} Comm. Math. Phys., 68, 209–243 (1979).
\bibitem{StS} M. Grillakis, J. Shatah, W. Strauss: \textit{Stability theory of solitary waves in the presence of symmetry, I} 
Journal of Functional Analysis 74, (1987).
\bibitem{PSF} P. Hajłasz: \textit{Functional Analysis} Lecture notes. 
\bibitem{Ir} I. D. Iliev, K. P. Kirchev:\textit{ Stability and Instability of Solitary waves for
one-dimensional singular Schrödinger equations,} Differential and Integral Eqs.
6, 685–703 (1993).
\bibitem{Arghir} R. Ignat, L. Nguyen, V. Slastikov, A. Zarnescu: \textit{Stability of the Melting Hedgehog in the Landau–de Gennes Theory of Nematic Liquid Crystals} Arch. Rat. Mech. Anal. 215, (2014) .
\bibitem{Jea} L. Jeanjean, S. Le Coz:\textit{ An existence and stability result for standing waves of
nonlinear Schrödinger equations} Adv. Differential Equations 11, 813–840 (2006).
\bibitem{KT} P. Kfoury, S. Le Coz, T.-P. Tsai :\textit{Analysis of stability and instability for standing waves of the double power one dimensional nonlinear Schrödinger equation.} Comptes Rendus. Mathématique, (2022). 
\bibitem{Kom} H. Kikuchi: \textit{Existence and stability of standing waves for Schrodinger-Poisson-Slater equation} Adv. Nonlinear Stud.  7, 403--437 (2007).
\bibitem{lib} Y. S. Kivshar, G.P. Agrawal: \textit{Optical Solitons. From Fibers to Photonic Crystals} Academic Press, San Diego (2003).
\bibitem{K} M. K. Kwong:\textit{ Uniqueness of positive solutions of $\Delta u - u + u^p=0$ in $\R^n$} Arch. Ration. Mech. Anal., 105, 243–266 (1989).
\bibitem{Lio} P.L. Lions: \textit{The concentration-compactness principle in the calculus of variations. The locally compact case. I} Ann. Inst. H. Poincaré Anal. Non Linéaire 1, 109–145 (1984).
\bibitem{Luo} X. Luo:\textit{ Stability and multiplicity of standing waves for the inhomogeneous NLS equation with a harmonic potential} Nonlinear Anal. Real World Appl.  45, 688--703, (2019).
\bibitem{Ne1} Z. Nehari: \textit{On a class of nonlinear second-order differential equations} Transactions of the American Mathematical Society (95) 101-123 (1960).
\bibitem{Ne2} Z. Nehari: \textit{Characteristic values associated with a class of nonlinear second-order differential equations} Acta Math. (105) 141–175 (1961).
\bibitem{OO} M. Ohta:\textit{ Instability of solitary waves for nonlinear Schrödinger equations of derivative type}
SUT J. Math. 50, 399–415 (2014).
\bibitem{aaa} P. Panayotaros, T.R. Marchant: \textit{Solitary waves in nematic liquid crystals} Physica D
268, 106–117 (2014).
\bibitem{as1} M. Peccianti, G. Assanto:  \textit{Nematicons} Physics Reports 516, 147–208 (2012).
\bibitem{Mon} S. Saks: \textit{Theory of the Integral} Dover Publications, New York, (1964).
\bibitem{S}J. Shatah: \textit{Stable standing waves of nonlinear Klein-Gordon equations} Comm. Math. Phys.
91, 313–327 (1983).
\bibitem{SS} J. Shatah, W. Strauss: \textit{Instability of nonlinear bound states} Comm. Math. Phys. 100, 173–190 (1985).
\bibitem{SW1} N. Shioji, K. Watanabe: \textit{A generalized Pohožaev identity and uniqueness of positive radial solutions of
$\Delta u + g(r )u + h(r )u ^p = 0$} J. Differ. Equ. 255, 4448–4475 (2013).
\bibitem{SW2} N. Shioji, K. Watanabe:\textit{ Uniqueness and nondegeneracy of positive radial solutions of ${\rm div}(\rho\nabla u)+\rho(-gu+hu^p)=0$} Calc. Var. Partial Differential Equations  55 (2016).
\bibitem{Rep2} C. A. Stuart :\textit{Lectures on the orbital stability of standing waves and application to the nonlinear Schrödinger equation} Milan J. Math. 76, 329–399 (2008).
\bibitem{Nena} A. Szulkin, T. Weth: \textit {Ground state solutions for some indefinite problems} Journal of Functional Analysis (257) 3802-3822 (2009)
\bibitem{Nene} A. Szulkin, T. Weth: \textit{The method of Nehari manifold} Lecture Notes 
(2010).
\bibitem{W} M. I. Weinstein: \textit{Modulational stability of ground states of nonlinear Schrödinger equations} Siam J. Math. Anal. 16, 472–491 (1985).
\bibitem{W2} M. I. Weinstein: \textit{Lyapunov stability of ground states of nonlinear dispersive evolution equations} Comm. Pure Appl. Math. 39, 51–67 (1986).
\bibitem{WC} M. I. Weinstein: \textit{Solitary waves of nonlinear dispersive evolution equations with critical power nonlinearities.} Journal of Differential Equations, 69(2), 192–203 (1987). 
\bibitem{Y}E. Yanagida: \textit{Uniqueness of positive radial solutions of $\Delta u+g(r)u+h(r)u^p= 0$
in $\R^n$} Arch. Rat. Mech. Anal. 115, 257–274 (1991).
\end{thebibliography}
\end{document}